\documentclass[11pt]{amsart}

\usepackage{Preamble}

\title{Splitting the Goldman-Turaev Lie bialgebra along a simple separating curve}

\author{Liam Ashton}\author{Manuel Rivera}
\address{Department of Mathematics \\ 150 N University St \\ Purdue University \\ West Lafayette, IN \\ 47907}
\email{lashton@purdue.edu}

\address{Department of Mathematics \\ 150 N University St \\ Purdue University \\ West Lafayette, IN \\ 47907}
\email{manuelr@purdue.edu}

\date{\today}

\begin{document}

\begin{abstract}
We introduce the notion of a double Lie bimodule, consisting of a Lie bialgebra $L$, a Lie $L$-bimodule $P$, and a double Lie bracket on $P$ satisfying suitable compatibility conditions. We describe an algebraic construction that combines two double Lie bimodules into a Lie bialgebra. As a geometric application, we show that the Goldman–Turaev Lie bialgebra of a surface with boundary, together with the Kawazumi–Kuno intersection operations on the linear span of the set of homotopy classes of paths between pairwise distinct boundary points, may be assembled into a double Lie bimodule. We prove a decomposition theorem for the Goldman–Turaev Lie bialgebra of an oriented surface in terms of the double Lie bimodules corresponding to the two surfaces with boundary obtained by cutting along a simple separating closed curve.
\end{abstract}

\maketitle

\section{Introduction}
\label{Sec:Introduction}
The Goldman–Turaev (GT) Lie bialgebra is a fundamental algebraic structure associated with an oriented surface. It is defined on the linear span of the set of free homotopy classes of closed curves modulo the class of the constant curve. Goldman constructed a Lie bracket in \cite{goldman} by considering transversal intersections among pairs of curves, while Turaev later defined a Lie cobracket in \cite{turaevcobracket} that records the transverse self-intersections of a curve. Together, this bracket and cobracket define a \textit{Lie bialgebra}, namely, the cobracket satisfies the Leibniz rule in the context of Lie algebras (\ref{liebialgebras}). This Lie biaglebra satisfies an additional \textit{involutivity} condition saying that the cobracket followed by the bracket is zero \cite{Cha04}. The GT Lie bialgebra plays a central role in low-dimensional topology, Teichmüller theory, and the study of moduli spaces, and is closely related to Poisson geometry, quantum topology, and skein algebras.

In this article, we introduce an algebraic structure, which we coin as a \textit{(involutive) double Lie bimodule}, consisting of a (involutive) Lie bialgebra $L$, a \textit{Lie $L$-bimodule} $P$, and a compatible \textit{double Lie bracket}
\[
[[-,-]] \colon P\otimes P \longrightarrow P\otimes P.
\]
The main geometric example of an involutive double Lie bimodule arises by considering homotopy classes of loops and paths on a surface with non-empty boundary. Our main result is then a Seifert-Van Kampen-type theorem expressing the GT Lie bialgebra of an oriented surface in terms of an algebraic construction that assembles the involutive double Lie bimodules corresponding to any two surfaces with boundary obtained by cutting the original surface along a separating simple closed curve. Our motivation is twofold. First, we seek to extend structural results about the GT Lie bialgebra and self-intersections of curves in surfaces with non-empty boundary to the closed case; see \cite{Cha04, CL12, CK16}. Second, we aim to develop analogous decomposition results for skein algebras of links in cylinders  over surfaces via quantization techniques extending the relation between skein theory and the GT Lie bialgebra established in \cite{turaevcobracket}.

Our results build upon operations considered by Kawazumi and Kuno in \cite{mappingclassgroups} and observations of Massuyeau and Turaev in \cite{quasipoisson} regarding intersection-type operations on paths that start and end at the boundary. Given an oriented surface $\Sigma$ with non-empty boundary $\partial \Sigma$ denote the associated GT Lie bialgebra by $L^\circ(\Sigma)$. Inspired by Turaev's work on intersections of loops on surfaces, \cite{Tu78}, Kawazumi and Kuno describe a \textit{Lie} $L^\circ(\Sigma)$-\textit{bimodule} structure on $P(\Sigma)(a,b)$, the linear span of homotopy classes of paths in $\Sigma$ that start at $a \in \partial \Sigma$ and end at $b \in \partial \Sigma$. This structure consists of an $L^\circ(\Sigma)$-action and $L^\circ(\Sigma)$-coaction operations on $P(\Sigma)(a,b)$ satisfying suitable compatibility conditions. We call this structure the KK-\textit{bimodule}. Given $a,b,c,d \in \partial  \Sigma$, Kawazami and Kuno also describe a two-to-two operation 
\[ [[-,-]] \colon P(\Sigma)(a,b) \otimes P(\Sigma)(c,d) \to P(\Sigma)(a,d) \otimes P(\Sigma)(c,b) \]
(denoted by $\kappa$ in \cite{mappingclassgroups}) by appropriately considering transversal intersections between pairs of paths. Massuyeau and Turaev prove that a version of this operation, defined when $a=b=c=d$, together with the concatenation product gives rise to a \textit{double quasi-Poisson algebra}, a version of a double Poisson algebra (see \cite{doublepoissonalgebras}) for which the double Jacobi identity is satisfied up to a correction term. In the present article, we describe how an honest double Lie bracket, compatible with the Lie bimodule structures, may be obtained by considering the operation only for $4$-tuples of pairwise \textit{distinct} boundary points. Before stating our geometric results, we briefly describe the algebraic notions and constructions. 

A \textit{double Lie algebra} is a vector space (or module over a commutative ring) $P$ equipped with a map $[[-,-]] \colon P \otimes P \to P \otimes P$, called the \textit{double Lie bracket}, satisfying versions of skew symmetry and the Jacobi identity (\ref{doubleliealgebras}). If $L$ is a (involutive) Lie bialgebra, a \textit{(involutive) Lie $L$-bimodule} structure on the vector space $P$ consists of a Lie action $\rho \colon L \otimes P \to P$ and a Lie coaction $\sigma \colon P \to L \otimes P$ that are appropriately (involutively) compatible. Furthermore, if $P$ is equipped with a double Lie bracket $[[-,-]]$, a (involutive) Lie $L$-bimodule structure on $P$ is called a \textit{(involutive) double Lie $L$-bimodule} if an additional set of  compatibility conditions between $\rho$, $\sigma$, and $[[-,-]]$ are satisfied (\ref{doubleliemodules} \ref{doubleliecomodules}, \ref{doubleliebimodule}). We define these notions in the context of linear quivers\footnote{A linear quiver $P$ consists of a set of objects $I$ and a collection of vector spaces $\{P(a,b)\}_{(a,b) \in I \times I}$.}. The compatibility conditions are required so that a double Lie bimodule naturally gives rise to an involutive Lie bialgebra as follows.

\begin{thm}[Theorem \ref{liebialgebra}] \label{intro1} Any (involutive) double Lie bimodule $(L,P)$ canonically determines a (involutive) Lie bialgebra structure on $L \oplus Cyc(P)$, where $Cyc(P)$ is an appropriately defined vector space of cyclic monomials on $P$. 
\end{thm}

Given two linear quivers $P_1$ and $P_2$ over the same set of objects, we define a vector space $P_1 \boxtimes P_2$, called the \textit{weaving product}, defined in terms of appropriately cyclic monomials that alternate between elements of $P_1$ and $P_2$ (\ref{weaving}). Our main algebraic result is the following.

\begin{thm}[Theorem \ref{weavingbialgebra}]\label{intro2} Any pair of (involutive) double Lie bimodules $(L_1,P_1)$ and $(L_2,P_2)$, cannonically determines a (involutive) Lie bialgebra structure on \[L_1 \oplus L_2 \oplus (P_1 \boxtimes P_2).\]
\end{thm}

The main geometric example of a double Lie bimodule, to which we apply the previous theorem, is the following. 

\begin{thm}[Theorem \ref{doubleLiebimodulethm}]\label{intro3}
Let $\Sigma$ be an oriented surface with non-empty boundary $\partial \Sigma$, and let $L^\circ(\Sigma)$ denote its Goldman–Turaev Lie bialgebra. For $a,b \in \partial \Sigma$, let $P(\Sigma)(a,b)$ be the linear span of homotopy classes of paths from $a$ to $b$, and set
\[
P(\Sigma)= \bigoplus_{(a,b)\in \partial \Sigma \times \partial \Sigma} P(\Sigma)(a,b).
\]
The Kawazumi–Kuno Lie $L^\circ(\Sigma)$-bimodule structure on $P(\Sigma)$, together with the map
\[
[[-,-]] \colon P(\Sigma)\otimes P(\Sigma)\to P(\Sigma)\otimes P(\Sigma)
\]
defined as the intersection operation on paths between pairwise distinct boundary points and the zero map otherwise, endows the pair $(L^\circ(\Sigma),P(\Sigma))$ with the structure of an involutive double Lie bimodule.
\end{thm}
Suppose now that $\Sigma$ is an oriented surface (with possibly empty boundary) and $\alpha$ a simple separating curve in $\Sigma$. Let $\Sigma_{1,\alpha}$ and $\Sigma_{2,\alpha}$ be the two oriented surfaces with boundary obtained by cutting $\Sigma$ along $\alpha$. Denote by $P(\Sigma_{i,\alpha})$ the linear span of the homotopy classes of paths in $\Sigma_{i,\alpha}$ with endpoints in the boundary curve $\alpha$. Our final main result, which is reminiscent of the Seifert-Van Kampen theorem, is the following.

\begin{thm}[Theorems \ref{gluingmap} and \ref{kerofvarphi}]\label{intro4}
There is a surjective map of Lie bialgebras
\[ \varphi \colon L^{\circ}(\Sigma_{1,\alpha}) \oplus L^{\circ}(\Sigma_{2,\alpha}) \oplus (P(\Sigma_{1,\alpha}) \boxtimes P(\Sigma_{2,\alpha}) ) \to L^{\circ}(\Sigma)\]
given by concatenating paths into loops. Furthermore, the kernel of $\varphi$ is given explicitly by the relations (11)-(15) in \ref{kerofvarphi}.
\end{thm}

This theorem shows that all homotopy classes of non-constant loops on an oriented surface, together with the bracket and cobracket operations, may be reconstructed in terms of loops on both sides of a separating curve together with cyclic words of paths alternating between each component. Seifert-Van Kampen style arguments have been used before to study the Goldman Lie algebra, see \cite{Cha04} and \cite{KK12}, but, to our knowledge, not for the GT Lie bialgebra.   

This paper is organized as follows. In Section 2, we review the necessary algebraic background and introduce the new algebraic notions and constructions used throughout the paper. This section is entirely algebraic. The proofs of Theorems \ref{intro1} and \ref{intro2}, which consist of explicit verifications of the defining identities, are collected in the Appendix. In Section 3, we present our main geometric example of a double Lie bimodule and prove Theorem \ref{intro3}. Finally, in Section 4, we establish the decomposition theorem, Theorem \ref{intro4}.

 \subsection{Acknowledgments} MR acknowledges support by NSF Grant DMS 2405405. Both authors would like to thank Nariya Kawazumi and Yusuke Kuno for suggesting relevant references and pointing out that the compatibility between the Lie action and the double Lie bracket (Section 3) was stated without proof in Section 5.2 of \cite{KK16} and that an equivalent formulation is proven in Lemma 7.4 of \cite{MT13}.

\section{Double Lie algebras and double Lie Bimodules}
\label{Sec:DoubleBracket}
In this section, we discuss the algebraic structures underlying our geometric constructions. The proofs of the four main theorems of this section (\ref{liebracket}, \ref{liecobracket}, \ref{liebialgebra}, and \ref{weavingbialgebra}) are tedious computations that can be found in the Appendix (Section \ref{Sec:Appendix}).

Our constructions and results hold not only in the context of vector spaces over a field but more generally for modules over fixed a commutative unital ring $\Bbbk$. We shall write $\otimes=\otimes_{\Bbbk}$ for the tensor product of $\Bbbk$-modules.

\subsection{Lie bialgebras} \label{liebialgebras}

Let $L$ be a $\Bbbk$-module and define linear maps $\tau_{(312)} \colon L \otimes L \otimes L \longrightarrow L \otimes L \otimes L$ and $\tau_{(12)} \colon L \otimes L \longrightarrow L \otimes L$ by
\[
\tau_{(312)}(u \otimes v \otimes w) = w \otimes u \otimes v \qquad \text{and} \qquad \tau_{(12)} (v \otimes w) = w \otimes v
\]
for any $u,v,w \in L$. A \textit{Lie algebra} structure on $L$ consists of a linear map $[-,-] \colon L \otimes L \longrightarrow L$, called the \textit{Lie bracket}, satisfying
\begin{enumerate}
\item \textit{skew symmetry}
\[[-,-] \circ \tau_{(12)} = - [-,-],\]
and
\item the \textit{Jacobi identity}
\[[-,-] \circ (\mathrm{Id} \otimes [-,-])\circ (\mathrm{Id} + \tau_{(312)} + \tau_{(312)}^2) = 0.\]
\end{enumerate}
A \textit{Lie coalgebra} structure on $L$ consists of a linear map $\Delta \colon L \longrightarrow L \otimes L$, called the \textit{Lie cobracket}, satisfying
\begin{enumerate}
\item \textit{skew co-symmetry}
\[\tau_{(12)} \circ \Delta = -\Delta,\]
and
\item the \textit{co-Jacobi identity}
\[
(\mathrm{Id} + \tau_{(312)} + \tau_{(312)}^2)\circ (\mathrm{Id} \otimes \Delta)\circ \Delta = 0.
\]
\end{enumerate}
The triple $(L, [-,-], \Delta)$ is called a \emph{Lie bialgebra} if $(L, [-,-])$ is a Lie algebra, $(L, \Delta)$ is a Lie coalgebra and, for every $v, w \in L$, the following compatibility holds
\[
\Delta[ v, w] = [ \Delta v, w] + [ v, \Delta w],
\]
where
\[
[ v, w \otimes u ] = -[ w \otimes u, v ] = [ v, w] \otimes u + w \otimes [v, u].
\]
A Lie bialgebra $(L, [-,-], \Delta)$ is said to be \emph{involutive} if $[-,-] \circ \Delta = 0$.

\subsection{Double Lie algebras}\label{doubleliealgebras}

\begin{dff}\label{doubleliealgebra}
  A \textit{double Lie algebra} is a $\Bbbk$-module $P$ together with a linear map
  \[[[-,-]] \colon P\otimes P \to P \otimes P,\] called the \textit{double Lie bracket}, satisfying
  \begin{enumerate}
    \item \textit{double skew-symmetry}
    \begin{equation*} 
      [[m, n]] = -\tau_{(12)}[[n, m]],
    \end{equation*} 
    and
    
    \item the \textit{double Jacobi identity}
    \begin{equation*}
      \langle m, [[n, p]]\rangle + \tau_{(123)}\langle n, [[p, m]]\rangle + \tau_{(132)}\langle p, [[m, n]]\rangle = 0,
    \end{equation*}
  for all $m,n,p \in P$,  where $\langle-,-\rangle : P \otimes P^{\otimes 2} \to P^{\otimes3}$ is given by 
  \[\langle m, n \otimes p\rangle = [[m, n]]^1 \otimes [[m, n]]^2 \otimes p\] writing \[[[m,n]]=[[m, n]]^1 \otimes [[m, n]]^2\] (using Sweedler type notation) and $\tau_{\iota}$ denotes the linear map determined by applying the permutation $\iota$ to tensor factors. 
  \end{enumerate}
\end{dff}

Double Lie brackets were introduced in \cite{doublepoissonalgebras} in the context when $P$ is equipped with an additional associative algebra structure satisfying a Poisson type compatibility; the resulting structure was coined as a \textit{double Poisson algebra}. The tensor algebra of a double Lie algebra has an induced double Poisson algebra structure, see Remark 2.11 in \cite{doublevertex}. See also \cite{doublealgebras}. 

 A $\Bbbk$-\textit{quiver} (or \textit{linear quiver}) $P$ consists of the data of a set of objects $I$ and, for each pair $(a,b) \in I \times I$, a $\Bbbk$-module $P(a,b)$. For any $\Bbbk$-quiver $P$, denote \[P^{\oplus}=\bigoplus_{(a,b)\in I\times I}P(a,b).\]
We define a double Lie bracket on a $\Bbbk$-quiver $P$ to be a collection of maps
\[[[-,-]]_{a,b,c,d} : P(a,b) \otimes P(c,d) \to P(a,d) \otimes P(c,b),\]
where $(a,b,c,d)$ is a $4$-tuple of pairwise distinct elements in $I$, such that the induced map
 \[[[-,-]] \colon P^{\oplus} \otimes P^{\oplus} \to P^{\oplus} \otimes P^{\oplus}, \] given by declaring $[[-,-]]_{a,b,c,d}=0$  whenever the elements $(a,b,c,d)$ are not pairwise distinct, is a double Lie bracket. Declaring the bracket to be zero when $(a,b,c,d)$ are not pairwise distinct will be important for our particular application. 

\subsection{Double Lie modules}\label{doubleliemodules} Suppose that $(L, [-,-])$ is a Lie $\Bbbk$-algebra and $P$ a $\Bbbk$-quiver with set of objects $I$. A \textit{left Lie $L$-module structure on $P$} consists of a collection of linear maps \[\rho_{a,b} : L \otimes P(a,b) \to P(a,b),\] for $(a,b) \in I \times I$, making each $\Bbbk$-module $P(a,b)$ into a left Lie module; namely, for any $\alpha, \beta \in L$ and $m \in P$,
\[\rho_{a,b}([\alpha, \beta], m) = \rho_{a,b}(\alpha,\rho_{a,b}(\beta,m)) - \rho_{a,b}(\beta, \rho_{a,b}(\alpha, m)).\] 

We can make $P(a,b)$ into a right Lie module by declaring the right action $\ol\rho_{a,b} : P(a,b) \otimes L \to P(a,b)$ to be
\[\ol\rho_{a,b}(m, \alpha) = -\rho_{a,b}(\alpha, m).\]

The collection of maps $\{\rho_{a,b}\}_{(a,b) \in I \times I}$ induces a left Lie module structure $\rho : L \otimes P^{\oplus} \to P^{\oplus}$ along with an associated right Lie module structure $\ol{\rho} : P^\oplus \otimes L \to P^\oplus$ 

Now suppose $P$ is furthermore equipped with a double Lie bracket $[[-,-]]$. 
We say $(L,P)$ is a \textit{double Lie $L$-module} if $P$ is a left Lie $L$-module and the following diagram commutes
  \begin{center}\begin{tikzcd}
  L \otimes P^{\oplus} \otimes P^{\oplus} \arrow[rr, "\widetilde{\rho}"] \arrow[d, "{1 \otimes [[-,-]]}"'] &  & P^{\oplus} \otimes P^{\oplus} \arrow[d, "{[[-,-]]}"']                     \\
  L \otimes P^{\oplus} \otimes P^{\oplus} \arrow[rr, "\widetilde{\rho}"]
  &  & P^{\oplus} \otimes P^{\oplus} 
  \end{tikzcd}\end{center}
  where $\widetilde{\rho} = \rho \otimes 1 + (1 \otimes \rho) \circ \tau_{(12)}$ and $\tau_{12} : L \otimes P^{\oplus} \otimes P^{\oplus} \to P^{\oplus} \otimes L \otimes P^{\oplus}$ is the map that swaps the first two coordinates.

We may obtain a Lie algebra from the data of a double Lie module $(L,P)$ as follows. Let
\begin{equation}\label{CycM}Cyc(P) \coloneqq \bigoplus_{n=2}^\infty \bigoplus_{\{ a_1, \dots, a_n \in I | a_i \neq a_j \}} P(a_1,a_2) \otimes P(a_3, a_4) \otimes \cdots \otimes P(a_{n-1},a_n) /\sim\end{equation}    
where $\sim$ is the relation $(f_1, \dots, f_n) \sim (f_n, f_1, \dots, f_{n-1})$ and $W = L \oplus Cyc(P)$. 
Define \[\{-,-\} : W \otimes W \to W\] by 
\begin{align*}
  \{(\alpha; f_1, \dots, f_n), (\beta; g_1, \dots, g_m)\} &= ([\alpha, \beta]; 0)\\
  &+ \sum_{j=1}^m (0; g_1, \dots, \rho(\alpha, g_j), \dots, g_m)\\
  &+ \sum_{i=1}^n (0; f_1, \dots, \ol{\rho}(f_i, \beta), \dots, f_n)\\
  &+ \sum_{i=1}^n\sum_{j=1}^n (0; [[f_i, g_j]]^1, g_{j+1}, \dots, g_{j-1}, [[f_i, g_j]]^2, f_{i+1}, \dots, f_{i-1}).
\end{align*}


\begin{thm}\label{liebracket}
  The map $\{-,-\} : W \otimes W \to W$ is a Lie bracket.
\end{thm}

  Skew symmetry follows from the skew symmetry of the maps $[-,-]$ and $[[-,-]]$ and the relation between $\rho$ and $\ol\rho$. The Jacobi identity follows from the double jacobi and compatibility with $\rho$ of $[[-,-]]$ and the fact that $\rho, \ol\rho$ are Lie module actions. A full detailed proof is written out in the appendix.

\subsection{Double Lie comodules}\label{doubleliecomodules}
Suppose that $(L, \delta)$ is a Lie $\Bbbk$-coalgebra and $P$ a $\Bbbk$-quiver with set of objects $I$. A \textit{left Lie $L$-comodule structure on $P$} consists of a collection of linear maps 

\[\sigma_{a,b} : P(a,b) \to L \otimes P(a,b)\] for $(a,b) \in I \times I$, making each $\Bbbk$-module $P(a,b)$ into a left Lie comodule; namely, for each $m \in P(a,b)$,

\[(\delta \otimes 1)(\sigma_{a,b}(m)) = (1 \otimes \sigma_{a,b})\sigma_{a,b}(m) - (\tau_{(12)} \otimes 1)(1 \otimes \sigma_{a,b})\sigma_{a,b}(m)\]

We can make $P(a,b)$ into a right Lie comodule by declaring the right coaction $\ol{\sigma}_{a,b} : P(a,b) \to P(a,b) \otimes L$ to be
\[\ol{\sigma}_{a,b}(m) = -\tau_{(12)}\circ \sigma_{a,b}(m).\]
The collection of maps $\{ \sigma_{a,b}\}_{(a,b) \in I \times I}$ induces a left Lie comodule structure $\sigma \colon P^{\oplus} \to L \otimes P^{\oplus}$ and an associated right Lie comodule structure $\ol{\sigma} : P^{\oplus} \to P^{\oplus} \otimes L$

Now suppose $P$ is furthermore equipped with a double Lie bracket $[[-,-]]$. We say $(L,P)$ is a \textit{double Lie $L$-comodule} if $P$ is a left Lie $L$-comodule and the following diagram commutes

\begin{center}\begin{tikzcd}
  P^{\oplus} \otimes P^{\oplus}  \arrow[rr, "\widetilde{\sigma}"] \arrow[d, "{[[-,-]]}"'] &  & L \otimes P^{\oplus}  \otimes P^{\oplus}  \arrow[d, "1 \otimes {[[-,-]]}"']                     \\
  P^{\oplus}  \otimes P^{\oplus}  \arrow[rr, "\widetilde{\sigma}"]                                                                   &  & L \otimes P^{\oplus}  \otimes P^{\oplus}  
  \end{tikzcd}\end{center}

where $\widetilde{\sigma} = \sigma \otimes 1 + \tau_{(12)} \circ (1 \otimes \sigma)$. 

\bigskip

Letting $W= L \oplus Cyc(P)$ as before, define 
\[\Delta : W \to W \otimes W\]
by
\begin{align*}
  \Delta(\alpha; f_1, \dots, f_n) =\ &(\delta(\alpha)^1; 0) \otimes (\delta(\alpha)^2, 0)\\
  + \sum_{i=1}^n\ &(\sigma(f_i)^1; 0) \otimes (0; f_1, \dots, \sigma(f_i)^2, \dots, f_n) \\
  +\ &(0; f_1, \dots, \ol{\sigma}(f_i)^1, \dots, f_n) \otimes (\ol{\sigma}(f_i)^2; 0) \\
  + \sum_{i<j}^{n}\ & (0; [[f_i, f_j]]^2, f_{i+1}, \cdots, f_{j-1}) \otimes (0; [[f_i, f_j]]^1, f_{j+1}, \cdots, f_{i-1}) \\
  +\ & (0; [[f_j, f_i]]^2, f_{j+1}, \dots, f_{i-1}) \otimes (0; [[f_j, f_i]]^1, f_{i+1}, \dots, f_{j-1}). 
\end{align*}

\begin{thm}\label{liecobracket}
  The map $\Delta : W \to W \otimes W$ is a Lie cobracket. 
\end{thm}

  Co-skew symmetry is obvious from the co-skew symmetry of $\delta$ and $[[-,-]]$ and the relationship $\sigma = -\tau_{(12)} \circ \ol\sigma$. Co-Jacobi follows from the co-Jacobi of $\delta$, the double Jacobi of $[[-,-]]$, and the compatibility of $\sigma$ with the other maps. A full detailed proof is written out in the appendix. 

\subsection{Double Lie bimodules}

Suppose that $(L, [-,-], \delta)$ is a Lie $\Bbbk$-bialgebra and $P$ a $\Bbbk$-quiver with set of objects $I$. Following \cite{mappingclassgroups}, a \textit{Lie $L$-bimodule structure on $P$} consists of a Lie $L$-module structure and a Lie $L$-comodule structure on $P$ whose structure maps $\rho \colon L \otimes P^{\oplus} \to P^{\oplus}$ and $\sigma \colon P^{\oplus} \to L \otimes P^{\oplus}$
satisfy the following compatibility equation:
  \[\sigma(\rho(a, m)) = ([-,-] \otimes 1 - 1 \otimes \ol{\rho} \circ \tau_{(132)})(1 \otimes \sigma)(a, m) + (1 \otimes \rho)(\delta \otimes 1)(a, m).\]
  If, in addition, the maps satisfy 
  \[\rho(\sigma(m)) = 0\]
  then we say the Lie $L$-bimodule structure is \textit{involutive}. 

We say the data $(L, [-,-],\delta, P, \rho, \sigma, [[-,-]])$ is a \textit{(involutive) double Lie $L$-bimodule} if
\begin{enumerate} \label{doubleliebimodule}
  \item $(L, [,], \delta)$ is an (involutive) Lie bialgebra.
\item $(P,[[-,-]])$ is a double Lie algebra, 
    \item $(L, [-,-], \rho, P, [[-,-]])$ is a double Lie $L$-module,
    \item $(L, \delta, \sigma, P, [[-,-]])$ is a double Lie $L$-comodule, 
    \item $(L, [-,-], \delta, P, \sigma,\rho)$ is a (involutive) Lie $L$-bimodule, and
    \item  for all $m,n \in P^{\oplus}$ the following equation holds 
    \[\big[\big[ [[m,n]]^1, [[m,n]]^2 \big]\big] = \rho(\sigma(n)^1, m) \otimes \sigma(n)^2 + \ol{\sigma}(m)^1 \otimes \ol{\rho}(n,\ol{\sigma}(m)^2).\]

\end{enumerate}

\begin{thm}\label{liebialgebra}
If $(L, [-,-],\delta, P, \rho, \sigma, [[-,-]])$ is a (involutive) double Lie $L$-bimodule,  $(W=L \oplus Cyc(P), \{-,-\}, \Delta)$ is an (involutive) Lie bialgebra. 
\end{thm}

See the appendix for a detailed proof.

\subsection{The Weaving Product}\label{weaving}

\begin{dff}
  Given two $\Bbbk$-quivers $V_1$ and $V_2$ with the same set of objects $I$ denote
  \[V_1 \smallsmile V_2 \coloneqq \bigoplus_{n=1}^{\infty} \bigoplus_{ \{a_1, \dots, a_{2n} \in I | a_i \neq a_j\}} V_1(a_1, a_2) \otimes V_2(a_2, a_3) \otimes V_1(a_3, a_4) \otimes V_2(a_4, a_5) \otimes \cdots \otimes V_1(a_{2n-1}, a_{2n}) \otimes V_2(a_{2n}, a_1).\]
We write elements of $V_1 \smallsmile V_2$ as $(f_1^{a_1, a_2}, g_1^{a_2, a_3}, f_2^{a_3,a_4}, g_2^{a_4,a_5}, \dots  f_{n}^{a_{2n-1}, a_{2n}}, g_{n}^{a_{2n}, a_1})$ and remove the superscripts when context makes the chosen elements of $I$ clear. The \textit{weaving product} is defined by
  \[V_1 \boxtimes V_2 \coloneqq (V_1 \smallsmile V_2)/\sim\]
  where $\sim$ is the relation 
  \[(f_1^{a_1, a_2}, g_1^{a_2, a_3}, f_2^{a_3,a_4}, g_2^{a_4,a_5}, \dots  f_{n}^{a_{2n-1}, a_{2n}}, g_{n}^{a_{2n}, a_1}) \sim (f_2^{a_3,a_4}, g_2^{a_4,a_5}, \dots  f_{n}^{a_{2n-1}, a_{2n}}, g_{n}^{a_{2n}, a_1}, f_1^{a_1, a_2}, g_1^{a_2, a_3}).\]
\end{dff}

Suppose $(L_i, [-,-]_i,\delta_i, M_i, \rho_i, \sigma_i, [[-,-]]_i)$ is a double Lie $L_i$-bimodule for $i=1,2$ and let \[V= L_1 \oplus L_2 \oplus (M_1 \boxtimes M_2).\] Define $\{-,-\}_1: V \otimes V \to V$ by
  \begin{align*}
    \{(\alpha; \beta&; f_1, g_1, \dots, f_n, g_n), (\alpha'; \beta'; f_1', g_1', \dots, f_m', g_m')\}_1 \\
    &= ([\alpha, \alpha']_1; 0; 0)\\
    &+ \sum_{j=1}^m (0; 0; f_1', g_1', \dots, \rho_1(\alpha, f_j'), g_j', \dots, f_m', g_m')\\
    &+ \sum_{i=1}^n (0; 0; f_1, g_1, \dots, \ol{\rho}_1(f_i, \alpha'), g_i, \dots, f_n, g_n)\\
    &+ \sum_{i=1}^n\sum_{j=1}^m (0; 0; [[f_i, f_j']]_1^1, g_j', f_{j+1}', g_{j+1}', \dots, f_{j-1}', g_{j-1}', [[f_i, f_j']]_1^2, g_i, f_{i+1}, g_{i+1}, \dots, f_{i-1}, g_{i-1}).
  \end{align*}
Similarly, define $\{-,-\}_2 : V \otimes V \to V$ by  
  \begin{align*}
    \{(\alpha; \beta&; f_1, g_1, \dots, f_n, g_n), (\alpha'; \beta'; f_1', g_1', \dots, f_m', g_m')\}_2 \\
    &= (0;[\beta, \beta']_2; 0)\\
    &+ \sum_{j=1}^m (0; 0; f_1', g_1', \dots, f_j', \rho_2(\beta, g_j'), \dots, f_m', g_m')\\
    &+ \sum_{i=1}^n (0; 0; f_1, g_1, \dots, f_i, \ol{\rho}_2(g_i, \beta'), \dots, f_n, g_n)\\
    &+ \sum_{i=1}^n\sum_{j=1}^m (0; 0; f_{i}, [[g_i, g_j']]_2^1, f_{j+1}', g_{j+1}', \dots, f_{j-1}', g_{j-1}', f_j', [[g_i, g_j']]_2^2, f_{i+1}, g_{i+1}, \dots, f_{i-1}, g_{i-1})
  \end{align*}
Finally, we define a bracket
  \[\{-,-\} : V \otimes V \to V \quad \{-,-\} = \{-,-\}_1 + \{-,-\}_2.\]
We can similarly define a cobraket operation as follows. First, define $\Delta_1 : V \to V \otimes V$ 
by 
\begin{align*}
  \Delta_1(\alpha; \beta&; f_1, g_1, \dots, f_n, g_n) =(\delta_1(\alpha)^1; 0;0) \otimes (\delta_1(\alpha)^2; 0;0)\\
  + \sum_{i=1}^n\ &(\sigma_1(f_i)^1; 0;0) \otimes (0;0; f_1, g_1, \dots, g_{i-1}, \sigma_1(f_i)^2, g_i, \dots, f_n,g_n) \\
  +\ &(0; 0;f_1,g_1, \dots, g_{i-1},\ol{\sigma}_1(f_i)^1, g_i, \dots, f_n,g_n) \otimes (\ol{\sigma}_1(f_i)^2; 0;0) \\
  + \sum_{i<j}^{n}\ & (0;0; [[f_i, f_j]]_1^2, g_i, f_{i+1}, \cdots, f_{j-1},g_{j-1}) \otimes (0; 0;[[f_i, f_j]]_1^1, g_j, f_{j+1}, \cdots, f_{i-1},g_{i-1}) \\
  +\ & (0;0; [[f_j, f_i]]_1^2, g_j, f_{j+1}, \dots, f_{i-1},g_{i-1}) \otimes (0; 0;[[f_j, f_i]]_1^1, g_{i}, f_{i+1}, \dots, f_{j-1},g_{j-1}). 
\end{align*}
Define $\Delta_2 : V \to V \otimes V$ 
by 
\begin{align*}
  \Delta_2(\alpha; \beta&; f_1, g_1, \dots, f_n, g_n) =( 0; \delta_2(\beta)^1; 0) \otimes ( 0; \delta_2(\beta)^2; 0)\\
  + \sum_{i=1}^n\ &(0;\sigma_2(g_i)^1;0) \otimes (0;0; f_1, g_1, \dots, f_{i}, \sigma_2(g_i)^2, f_{i+1}, \dots, f_n,g_n) \\
  +\ &(0; 0;f_1,g_1, \dots, f_{i},\ol{\sigma}_2(g_i)^1, f_{i+1}, \dots, f_n,g_n) \otimes (0;\ol{\sigma}_2(g_i)^2;0) \\
  + \sum_{i<j}^{n}\ & (0;0; [[g_i, g_j]]_2^2, f_{i+1},g_{i+1} \cdots, f_{j-1},g_{j-1}) \otimes (0; 0;[[g_i, g_j]]_2^1, f_{j+1},g_{j+1} \cdots, f_{i-1},g_{i-1}) \\
  +\ & (0;0; [[g_j, g_i]]_2^2, f_{j+1},g_{j+1} \dots, f_{i-1},g_{i-1}) \otimes (0; 0;[[g_j, g_i]]_2^1, f_{i+1}, g_{i+1}, \dots, f_{j-1},g_{j-1}). 
\end{align*}
Finally, define a cobracket

  \[\Delta : V \to V \otimes V \quad \Delta = \Delta_1 + \Delta_2.\]

\begin{thm}\label{weavingbialgebra} If $(L_i, [-,-]_i,\delta_i, P_i, \rho_i, \sigma_i, [[-,-]]_i)$ is an involutive double Lie $L_i$-bimodule for $i=1,2$, then
$(V=L_1 \oplus L_2 \oplus (P_1 \boxtimes P_2), \Delta, \{-,-\})$ defines an involutive Lie bialgebra.
\end{thm}
See the appendix for a detailed proof.

\section{The Goldman-Turaev Lie Bialgebra and its Bimodule of Paths}
\label{Sec:GoldmanTuraevLieBilgebraAndModule}
\subsection{Goldman-Turaev Lie bialgebra} Denote by $[Map(S^1, \Sigma)]$ the set of free homotopy classes of oriented loops in $\Sigma$ and let $L=L(\Sigma)$ the free $\Bbbk$-module generated by $[Map(S^1, \Sigma)]$. Goldman constructed in \cite{goldman} a Lie bracket on $L$, which we now recall. Let $\alpha, \beta \in Map(S^1, \Sigma)$ be oriented immersions that intersect only at a finite number of transverse double points. For each $p \in \alpha \cap \beta$, we define $\alpha_p\beta \in Map(S^1, \Sigma)$ as the loop that goes around $\alpha$, starting and ending at $p$, then around $\beta$ starting and ending at $p$. Let $\epsilon(p; \alpha,\beta) \in \{\pm 1\}$ be the local intersection number of $\alpha$ and $\beta$ at $p$; namely, it is $+1$ if the ordered basis $(d\alpha_s(1), d\beta_t(1))$ of the tangent space $T_p\Sigma$ is in the class of the chosen orientation and $-1$ otherwise, where $s, t \in S^1$ are such that $\alpha(s)=p=\beta(t)$.

The Goldman bracket is defined as
\[\big[\alpha, \beta \big] = \sum_{p \in \alpha \cap \beta} \epsilon(p; \alpha, \beta) \alpha_p \beta.\]
Choosing appropriate representatives of free homotopy classes of oriented loops, this construction gives rise to a well defined Lie bracket $$[-,-] \colon L \otimes L \to L.$$

We also recall Turaev's cobracket on the quotient $L^\circ =L/\Bbbk\langle e\rangle$, where $e$ denotes the class of the constant loop \cite{turaevcobracket}. Let $\alpha \in Map(S^1, \Sigma)$ be an oriented immersion that intersects itself only at a finite number of transverse double points and let $I(\alpha) \subset \Sigma$ be the set of self intersection points. Given $p \in I(\alpha)$ there are $s < t \in [0,1]$ such that $\alpha(s) = \alpha(t) = p$. Define $\epsilon(p; \alpha) \in \{\pm1\}$ as the orientation of the basis $(d\alpha_s(1), d\alpha_t(1))$. Let $\alpha_{1,p}$ be the loop determined by $\alpha|_{[s, t]}$ and $\alpha_{2,p}$ the loop determined by $\alpha|_{[t, 1]} * \alpha|_{[0, s]}$. The Turaev cobracket is defined as
\[\delta(\alpha) = \sum_{p \in I(\alpha)} \epsilon(p; \alpha) (\alpha_{1,p} \otimes \alpha_{2,p} - \alpha_{2,p} \otimes \alpha_{1,p}).\]
Choosing an appropriate representative of a free homotopy class of oriented loops, this construction gives rise to a well defined Lie cobracket $$\delta \colon L^\circ \to L^\circ \otimes L^\circ.$$ The Goldman Lie bracket also descends to $L^\circ$ and it is compatible with the Turaev cobracket in the following sense.

\begin{thm}[\cite{turaevcobracket, Cha04}]
 The Goldman bracket and the Turaev cobracket endow $L^\circ$ with the structure of an involutive Lie bialgebra .
\end{thm} 

\subsection{Kawazumi-Kuno Lie bimodule} When $\Sigma$ is a surface with non-empty boundary, for every $a, b \in \bd\Sigma$, the free $\Bbbk$-module generated by homotopy classes of paths from $a$ to $b$ has the structure of a Lie module over Goldman's Lie algebra, which was originally described in \cite{mappingclassgroups} and we now recall. 

Let $f : ([0, 1], 0, 1) \to (\Sigma, a, b)$ be an an immersed path and  $\gamma \in Map(S^1, \Sigma)$ an immersed loop that intersects $f$ only at transverse double points. For any intersection poin $p \in f\cap\gamma$, we define $\gamma_p f: ([0,1], 0, 1) \to (\Sigma, a, b)$ as the path that starts at $a$, follows $f$ until the point $p$, then follows $\gamma$ around and back to $p$, then again follows $f$ until ending at $b$. Let $P(a,b)=P_{\Sigma}(a,b)$ be the free $\Bbbk$-module generated by homotopy classes of  maps $([0,1], 0, 1) \to (\Sigma, a, b)$. Define
\[\rho_{a,b}(\gamma, f) = \sum_{p \in \gamma \cap f} \epsilon(p; \gamma, f) \gamma_p f\]
where $\epsilon(p; \gamma, f)$ is once again the local intersection number of $\gamma$ with $f$. This operation gives rise to a well defined $L^\circ$-action $\rho_{a,b}: L^\circ \otimes P(a,b) \to P(a,b)$. See figure \ref{fig:action} for an example.

\begin{figure}
    \centering
    \input{Graphics/action}
    \caption{The action $\rho_{a,b}(\textcolor{Red}{\gamma}, \textcolor{RoyalBlue}{f})$.}
    \label{fig:action}
\end{figure}

Kawazumi and Kuno also describe a Lie $L^\circ$-comodule structure on $P(a,b)$. Suppose the path $f$ intersects itself only at transverse double points in $\Sigma \setminus \{a, b\}$. Let $I(f)$ be the set of self intersection points. Then, for $q \in I(f)$, let $s < t \in [0, 1]$ such that $f(s) = f(t) = q$. Then, we define $f_q \in L$ as the loop that follows the path $f|_{[s, t]}$ and $f_{\setminus q} \in P$ as the path that follows $f|_{[0, s] \cup [t, 1]}$ which is continuous because $f(s) = f(t)$. Define $\epsilon(q; f)$ as the orientation of the basis $\{df_s(1), df_t(1)\}$. Let 
\[\sigma_{a,b}(f) = \sum_{q \in I(f)} \epsilon(q; f) f_q \otimes f_{\setminus q}.\]
This operation gives rise to a well defined Lie $L^\circ$-coaction $\sigma_{a,b}: P(a,b) \to L^\circ \otimes P(a,b)$.

See figure \ref{fig:coaction} for an example.

\begin{figure}
    \centering
    \input{Graphics/coaction}
    \caption{The coaction $\sigma_{a,b}(\textcolor{Red}{f})$.}
    \label{fig:coaction}
\end{figure}

\begin{thm}[\cite{mappingclassgroups}] \label{KKinvolutiveLiebimodule}
For each $a,b \in \bd \Sigma$, the maps $\rho_{a,b}$ and $\sigma_{a,b}$ define an involutive Lie $L^{\circ}$-bimodule structure on $P(a,b)$.
\end{thm}

Taking the points of $\bd \Sigma$ as a set of objects, the collection $\{P(a,b)\}_{a,b \in \bd \Sigma}$ defines a $\Bbbk$-quiver and we obtain an induced Lie $L^\circ$-bimodule structure  on \[P=P(\Sigma)= \bigoplus_{a,b \in \bd\Sigma} P(a,b)\] with action and coaction denoted by $\rho :L^\circ \otimes P \to P$ and  $\sigma : P \to L^\circ \otimes P$, respectively. 

\subsection{Double Lie bimodule associated to a surface with boundary}

Following \cite[Section 4.2]{mappingclassgroups} we define a double bracket \[[[-,-]] : P(a,b) \otimes P(c,d) \to P(a,d) \otimes P(c,b),\] for any pairwise distinct points $a,b,c,d \in \bd \Sigma$, that is compatible with the $L^\circ$-bimodule structure on $P$. Let $f, g$ be immersed paths in $\Sigma$ from $a$ to $b$ and from $c$ to $d$, respectively, that intersect only at transverse double points in $\Sigma \setminus \bd\Sigma$. If $p \in f \cap g$, then we define $f_pg$ to be the the path that follows $f$ from $a$ to $p$ and then follows $g$ from $p$ to $d$. No matter which way we parametrize this concatenation, the operation is well defined on homotopy classes of paths. 

Let $\epsilon(p; f, g) \in \{\pm1\}$ be the local intersection number similar to above. Define
\begin{equation}\label{doublebracket}
[[f, g]] = \sum_{p \in f \cap g} \epsilon(p; f, g) f_pg \otimes g_pf.
\end{equation}
See figure \ref{fig:doublebracket} for an example.

\begin{figure}
    \centering
    \input{Graphics/double_bracket}
    \caption{The double bracket $[[\textcolor{Red}{f}, \textcolor{RoyalBlue}{g}]]$.}
    \label{fig:doublebracket}
\end{figure}

In \cite{mappingclassgroups}, this operation is also defined when $a,b,c,d$ are not necessarily distinct and a Poisson-type compatibility with the concatenation product is established. In fact, when $a=b=c=d$, Massuyeau and Turaev establish that a version of $[[-,-]]$ gives rise to an operation on the fundamental group algebra \[\Bbbk [\pi_1(\Sigma,a) ]\otimes \Bbbk [\pi_1(\Sigma,a)]  \to \Bbbk [\pi_1(\Sigma,a) ] \otimes \Bbbk [\pi_1(\Sigma,a)]\] defining a \textit{quasi-Poisson double algebra}, namely, a version of a Poisson double algebra in which the double Jacobi identity is satisfied up to a correction term \cite{quasipoisson}. 

Extend the operation $[[-,-]]$, defined in Equation \ref{doublebracket} above, to $P =  \bigoplus_{a,b \in \bd\Sigma} P(a,b)$
by declaring it to be zero if $(a,b,c,d)$ are not pairwise distinct so that the only possibly non-zero operation occurs between pairwise distinct points. This gives rise to a linear map
\begin{equation}
    [[-,-]]\colon P \otimes P \to P \otimes P,
\end{equation}
which we show below to be a double Lie bracket on $P$ compatible with the Lie $L^\circ$-bimodule structure. The Poisson compatibility with the concatenation product is no longer satisfied, but this is not relevant for our purposes. For the rest of this section, we will assume that differently-named points in $\bd\Sigma$ are pairwise distinct. The main result of this section is the following. 

\begin{thm}\label{doubleLiebimodulethm}
 The data $(L^\circ,[-,-],\delta, P, \rho,\sigma, [[-,-]])$ defines an involutive double Lie $L^\circ$-bimodule associated to any oriented surface $\Sigma$ with non-empty boundary. 
\end{thm}
\noindent 
\textit{Proof of Theorem \ref{doubleLiebimodulethm}.} The map $[[-,-]]$ is clearly skew-symmetric, so we proceed to prove the double Jacobi identity and (3), (4), and (6) in the definition of an involutive double Lie bimodule \ref{doubleliebimodule}. 
Note that condition (5) follows directly from Kawazumi and Kuno's Theorem \ref{KKinvolutiveLiebimodule}. The proof of these facts is organized through the a sequence of lemmas. For two paths $f$ and $g$ with same endpoints, $f \simeq g$ means $f$ an $g$ are homotopic through paths preserving the endpoints. In the context of loops, the symbol $\simeq$ will mean freely homotopic.

  \begin{lem}\label{samepath}
    Let $a, b, c, d, i, j \in \partial\Sigma$. Suppose that $f$ is a path in $\Sigma$ from $a$ to $b$, $g$ is a path from $c$ to $d$, and  $h$ is a path from $i$ to $j$. Suppose further that $f, g$, and $h$ intersect only at transverse double points and $t,s \in (0, 1)$ are such that $t<s$, $g(t) = p \in f \cap g$, and $g(s) = q \in g \cap h$. Then, 
    \[(g_qh)_pf \simeq g_pf, \quad h_qg \simeq h_q(f_pg), \mathrm{and~~} f_p(g_qh) \simeq (f_pg)_qh.\]
  \end{lem}

  \begin{proof}
  The path $g_pf$ is given by running $g$ from $c$ to $p$, then following $f$ from $p$ to $b$. The path $g_qh$ is given by running $g$ from $c$ to $q$, then following $h$ from $q$ to $j$. Since $p \in g((0, s))$, the path $(g_qh)_pf$ follows $g$ from $c$ to $p$ then follows $f$ from $p$ to $b$. Notably, it does not follow any part of $h$. Thus, the two paths differ by a reparameterization, so $(g_qh)_pf\simeq g_pf$. A similar argument shows that $h_qg \simeq h_q(f_pg)$. 

Finally, it is straightforward to check that $f_p(g_qh)$ and $(f_pg)_qh$ are homotopic, since concatenation of paths is associative up to reparameterization. 

See figures \ref{fig:lemma3.3.1_1} through \ref{fig:lemma3.3.1_3} for an example.

\begin{figure}
    \centering
    \includegraphics[width=.45\textwidth]{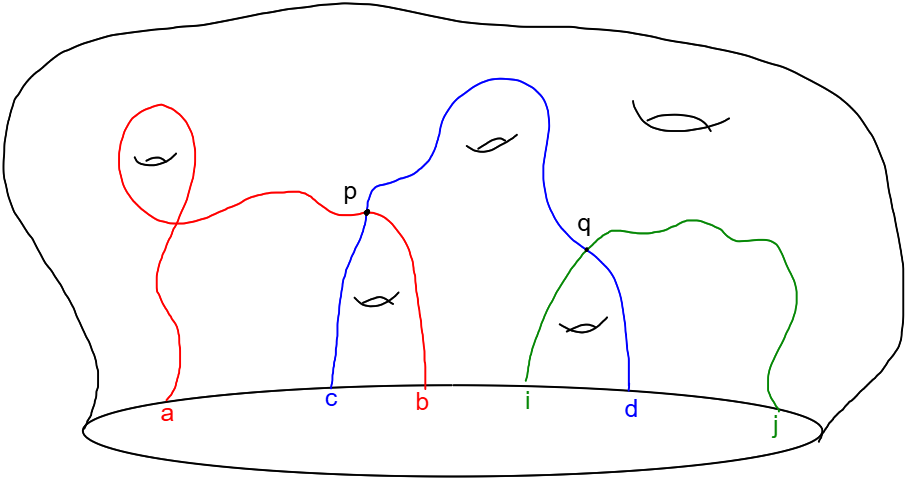}
    \caption{Applying lemma 3.3.1 to the paths $\textcolor{Red}{f}, \textcolor{RoyalBlue}{g}, \textcolor{OliveGreen}{h}$.}
    \label{fig:lemma3.3.1_1}
\end{figure}

\begin{figure}
    \centering
    \input{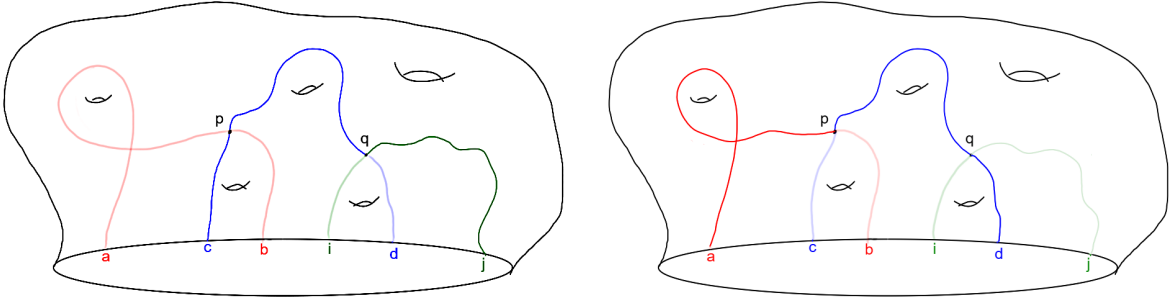}
    \caption{Comparing the paths $g_qh$ and $f_pg$}
    \label{fig:lemma3.3.1_2}
\end{figure}

\begin{figure}
    \centering
    \includegraphics[width=.45\textwidth]{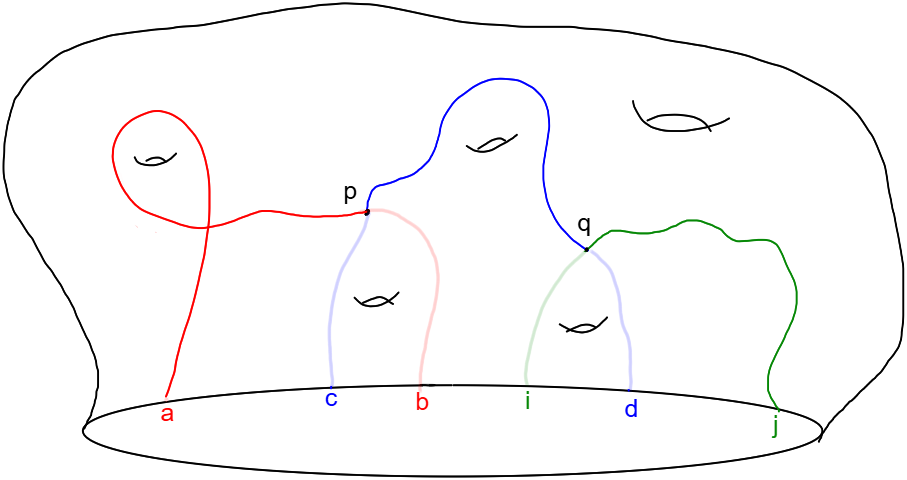}
    \caption{Showing $f_p(g_qh) = (f_pg)_qh$.}
    \label{fig:lemma3.3.1_3}
\end{figure}
  \end{proof}

  \begin{lem}\label{oppsign}
    Suppose that $f, g, h$ and $t, s, p, q$ are as in the previous lemma. Then, 
    \[\epsilon(p; f, g)\epsilon(q; h, f_pg) = -\epsilon(q; g, h)\epsilon(p; f, gh_q).\]
  \end{lem}

  \begin{proof}
    Since $p \in f \cap g$, then $\epsilon(p; f, gh_q) = \epsilon(p; f, g)$. Since $q \in g \cap h$, $\epsilon(q; h, f_pg) = \epsilon(q; h, g) = -\epsilon(q; g, h)$.
  \end{proof}

 We now prove $[[-,-]] \colon P \otimes P \to P \otimes P$ satisfies the double Jacobi identity. We have
  \begin{align*}
    &\{f, [[g, h]]\} + \tau_{(123)}\{g, [[h, f]]\} + \tau_{(132)}\{h, [[f, g]]\} \\
    &= \sum_{p \in g \cap h} \epsilon(p; g, h)\{f, g_ph \otimes h_pg\} \\
    &+ \sum_{p \in h \cap f} \epsilon(p; h, f)\tau_{(123)}\{g, h_pf \otimes f_ph\} \\
    &+  \sum_{p \in f \cap g} \epsilon(p; f, g)\tau_{(132)} \{h, f_pg \otimes g_pf\}\\
    &= \sum_{p \in g \cap h} \sum_{q \in f \cap g_ph} \epsilon(p; g, h) \epsilon(q; f, g_ph) f_q(g_ph) \otimes (g_ph)_qf \otimes h_pg \\
    &+ \sum_{p \in h \cap f} \sum_{q \in g \cap h_pf} \epsilon(p; h, f) \epsilon(q; g, h_pf) f_ph \otimes g_q(h_pf) \otimes (h_pf)_qg\\
    &+ \sum_{p \in f \cap g} \sum_{q \in h \cap f_pg} \epsilon(p; f, g) \epsilon(q; h, f_pg) (f_pg)_qh \otimes g_pf \otimes h_q(f_pg).\\
  \end{align*}
  Take a pair of intersection points $u,v$ such that $u \in f \cap g$ and $v \in g \cap h$, where $g(t) = u$ and $g(s) = v$. This pair appears in the above sum if and only if $t < s$, because if $s > t$, then $f$ intersects $h_vg$ at $u$, not $g_vh$. Similarly, $h$ intersects $g_uf$ at $v$, not $f_ug$. If $t < s$, the pair appears exactly twice in the above sum as the following terms: 
  \[\epsilon(v; g, h) \epsilon(u; f, gh_v)f_u(gh_v) \otimes (gh_v)_uf \otimes hg_v \mathrm{~and~} \epsilon(u; f, g) \epsilon(v; h, fg_u) (fg_u)_vh \otimes gf_u \otimes h_v(fg_u)\]
  By Lemmas (\ref{samepath}) and (\ref{oppsign}), these two terms cancel. The same can be said for when $u \in f \cap h$, $v \in h \cap g$ or $u \in g \cap f$ and $v \in f \cap h$. This covers every term in the above sum. 

  We now prove that $\rho$ and $[[-,-]]$ make $P$ into a double Lie $L^\circ$-module. We begin by stating two straightforward lemmas  similar in flavor to the previous two.

  \begin{lem}\label{samepathloop}
    Let $a, b, c, d \in \partial\Sigma$. Suppose that $f$ is a path from $a$ to $b$ and $g$ is a path from $c$ to $d$. Suppose further that $f$ and $g$ intersect only at transverse double points, and that $\gamma \in Map(S^1, \Sigma)$ intersects $f$ and $g$ only in $\Sigma \setminus\bd\Sigma$ at transverse double points. Let $t \neq s \in (0,1)$ $g(t) = p \in f \cap g$ and $g(s) = q \in \gamma \cap g$. Then, the following are true. 
    \begin{enumerate}
      \item If $t < s$, then 
      \[f_p(\gamma_qg) \simeq \gamma_q(f_pg) \mathrm{~and~} (\gamma_qg)_pf \simeq g_pf\]
      \item If $s < t$, then 
      \[f_p(\gamma_qg) \simeq f_pg \mathrm{~and~} (\gamma_qg)_pf \simeq \gamma_q(g_pf)\]
    \end{enumerate}
  \end{lem}
  \begin{proof}
    Since $t < s$, the path $f_p(\gamma_qg)$ is the path that follows $f$ from $a$ to $p$, then follows $g$ from $p$ to $q$, goes around $\gamma$, then once again follows $g$ from $q$ to $d$. The path $\gamma_q(f_pg)$ does the same with a different parametrization. Thus, they are in the same homotopy class. 

    The path $(\gamma_qg)_pf$ follows $g$ from $c$ to $p$, then follows $f$ from $p$ to $b$. Notice that it does not follow $\gamma$ at all. Thus, it is in the same homotopy class as $g_pf$. 

    For the second part, the path $f_p(\gamma_qg)$ follows $f$ from $a$ to $p$, then follows $g$ from $p$ to $d$. It does not contain the point $q$ because $s < t$ and thus does not interact with $\gamma$, so it is in the same homotopy class as $f_pg$. 
    
    The last claim follows since concatenation of paths is associative up to homotopy. In fact, the path $(\gamma_qg)_pf$ follows $g$ from $c$ to $q$, then goes around $\gamma$, then follows $g$ again from $q$ to $p$, then follows $f$ from $p$ to $b$. The path $\gamma_q(g_pf)$ does the same with a different parametrization. Thus, they are in the same homotopy class.
  \end{proof}

  \begin{lem}\label{twolooppoints}
    Suppose that $f, g,$ and $\gamma$ are as in the previous lemma. Let $s \neq t \in [0, 1)$ and let $\gamma(s) = p \in \gamma \cap f$ and $\gamma(t) = q \in \gamma \cap g$. Then, 
    \[(\gamma_p f)_qg \simeq f_p(\gamma_q g) \mathrm{~and~} g_q(\gamma_p f) \simeq (\gamma_qg)_pf.\]
  \end{lem}

  \begin{proof}
    The path $(\gamma_p f)_qg$ first follows $\gamma_p f$ from $a$ to $q$, which splits as following $f$ from $a$ to $p$ and $\gamma$ from $p$ to $q$, then follows $g$ from $q$ to $d$. The path $f_p(\gamma_qg)$ also follows $f$ from $a$ to $p$, then follows $\gamma_qg$ from $p$ to $d$, which is following $\gamma$ from $p$ to $q$ then $g$ from $q$ to $d$. Thus, the two paths differ by a reparameterization.

   Similarly, the path $g_q(\gamma_pf)$ follows $g$ from $c$ to $q$, then follows $\gamma_p f$ from $q$ to $b$, which is the same as following $\gamma$ from $q$ to $p$ then following $f$ from $p$ to $b$. The path $(\gamma_qg)_pf$ follows $g$ from $c$ to $q$, the $\gamma$ from $q$ to $p$, then $f$ from $p$ to $b$. Thus, the two paths are differ by a reparameterization.
  \end{proof}

We now establish the compatibility between $\rho$ and $[[-,-]]$.  Let $\gamma \in Map(S^1, \Sigma)$ and $a, b, c, d \in \bd\Sigma$. Suppose $f$ is a path from $a$ to $b$ and $g$ is a path from $c$ to $d$. Then, 
  \begin{align}
    \gamma \otimes f \otimes g \xmapsto{1 \otimes [[-,-]]}& \sum_{p \in f \cap g} \epsilon(p; f, g) \gamma \otimes f_pg \otimes g_pf\\
    \xmapsto{\widetilde{\rho}} &\sum_{p\in f \cap g}\sum_{q \in \gamma \cap f_pg} \epsilon(p;f,g)\epsilon(q;\gamma, f_pg)\gamma_q(f_pg) \otimes g_pf \\
    +& \sum_{p \in f \cap g}\sum_{q \in \gamma \cap g_pf} \epsilon(p; f,g)\epsilon(q; \gamma, g_pf) f_pg \otimes \gamma_q(g_pf)
  \end{align}
  and 
  \begin{align}
    \gamma \otimes f \otimes g \xmapsto{\widetilde{\rho}}& \sum_{q\in \gamma \cap f} \epsilon(q; \gamma, f) \gamma_q f \otimes g \\
    +& \sum_{q \in \gamma \cap g} \epsilon(q; \gamma, g) f \otimes \gamma_q g\\
    \xmapsto{[[-,-]]}&\sum_{q \in \gamma \cap f}\sum_{p \in \gamma f \cap g} \epsilon(q; \gamma, f)\epsilon(p; \gamma f, g) (\gamma_qf)_pg \otimes g_p(\gamma_qf) \\
    +& \sum_{q \in \gamma \cap g} \sum_{p \in f \cap \gamma g} \epsilon(q; \gamma, g) \epsilon(p; f, \gamma g) f_p(\gamma_qg) \otimes (\gamma_qg)_pf
  \end{align}
  Every pair of intersection points $u \in f \cap g$, $v \in g \cap \gamma$ appears in first sum once and the second sum once. If $g(t) = u$ and $g(s) = v$ with $t < s$, then by Lemma \ref{samepathloop}, we have
  \[\epsilon(u; f, g)\epsilon(v; \gamma, f_ug) \gamma_v(f_ug) \otimes g_uf \mathrm{~and~} \epsilon(v; \gamma, g)\epsilon(u; f, \gamma_vg) f_u(\gamma g_v) \otimes (\gamma_v g)_uf\]
  define the same element in $P \otimes P$ since $\epsilon(u; f, g) = \epsilon(u; f, \gamma_vg)$ and $\epsilon(v; \gamma, f_ug) = \epsilon(v; \gamma, g)$. 
  
  Lemma \ref{samepathloop} also proves the above is true for $s < t$ and for $u \in f \cap g$ and $v \in \gamma \cap f$ in either order. 

  There is one additional case we must consider: when $u \in \gamma \cap f$ and $v \in \gamma \cap g$. Notice that this pair of intersection points will not appear on any terms in the first sum because each of those must have a point in $f \cap g$. However, it appears twice in the second sum as 
  \[\epsilon(u; \gamma, f) \epsilon(v; \gamma_u f, g) (\gamma_u f)_vg \otimes g_v(\gamma_uf) \mathrm{~and~} \epsilon(v, \gamma, g) \epsilon(u; f, \gamma_v g) f_u(\gamma_vg) \otimes (\gamma_vg)_uf\]
  Because $v \in \gamma \cap g$, we have $\epsilon(v; \gamma f, g) = \epsilon(v; \gamma, g)$, and because $u \in \gamma \cap f$, $\epsilon (u; f, \gamma_vg) = \epsilon(u; f, \gamma) = -\epsilon(u; \gamma, f)$. We then see the signs on each term are opposite and, by Lemma \ref{twolooppoints}, the paths determine the same element in $P \otimes P$, so these terms cancel.
  
  Thus, the two sums above (equations (5)-(6) and (9)-(10)) match up term by term as elements in $P \otimes P$, proving the desired compatibility. 

  Next, we prove $\sigma$ and $[[-,-]]$ make $P$ into a double Lie $L^\circ$-comodule. Once again, we start with a few lemmas. 

  \begin{lem}\label{twopathself}
    Let $a, b, c, d \in \partial\Sigma$. Suppose that $f$ is a path from $a$ to $b$ and $g$ is a path from $c$ to $d$, such that that all self-intersections of $f$ and all intersections between $f$ and $g$ occur in $\Sigma \setminus \bd\Sigma$ at transverse double points. Let $p \in I(f)$ with $s < t \in (0,1)$ such that $f(s) = f(t) = p$ and $q \in f \cap g$ with $f(r) = q$ for some $r \in (0,1)$. 
    \begin{enumerate}
      \item If $r < s$, then 
      \[f_p \simeq (f_qg)_p, \quad (f_{\setminus p})_qg \simeq (f_qg)_{\setminus p}, \quad \mathrm{and~} g_q(f_{\setminus p}) \simeq g_qf.\]
      \item If $t < r$, then 
      \[f_p \simeq (g_qf)_p, \quad (f_{\setminus p})_qg \simeq f_qg, \quad \mathrm{and~} g_q(f_{\setminus p}) \simeq (g_qf)_{\setminus p}.\]
    \end{enumerate}
  \end{lem}

  \begin{proof}
    We prove (1). By definition, $f_p$ is the loop that follows $f$ from $f(s)$ to $f(t)$. Since $p \in I(f)$ and $r < s$, we have $p$ as a self intersection point of $f_qg$, and $(f_qg)_{p}$ follows $f$ from $f(s)$ to $f(t)$ as well with a different reparameterization. Thus both loops are homotopic. 

    The path $(f_{\setminus p})_qg$ follows $f$ from $a$ to $p$, skips over the portion of $f$ covered by $f_p$, then follows $f$ from $p$ to $q$, and follows $g$ from $q$ to $d$. The path $(f_qg)_{\setminus p}$ follows $f$ from $a$ to $p$ (since $r < s$), then skips over the portion covered by $f_p$, follows $f$ from $p$ to $q$, then follows $g$ from $q$ to $d$. Thus, they define the same path. 

    The path $g_q(f_{\setminus p})$ follows $g$ from $c$ to $q$, then, since $r < s$, follows $f$ from $q$ to $b$. This is the exact same description as for the path $g_qf$, thus, they are the same. 

    (2) follows from a similar argument. 
  \end{proof}

  \begin{lem}\label{selfsamesign}
    Let $f, g, p, q, r, s, t$ be as in the previous lemma. Then, $\epsilon(q, f, g) = \epsilon(q; f_{\setminus p}, g)$ and 
    \begin{enumerate}
      \item If $r < s$, then 
      \[\epsilon(p; f) = \epsilon(p; f_qg).\]
      \item If $t < r$, then 
      \[\epsilon(p; f) = \epsilon(p; g_qf).\]
    \end{enumerate}
  \end{lem}

  \begin{proof}
    Obvious.
  \end{proof}

  For the next lemma, we consider two intersection points between $f$ and $g$ such that after applying the double bracket to one intersection point, the other intersection point becomes a self intersection of one of the resulting loops. 

  \begin{lem}\label{twopathintersectcancel}
    Let $a, b, c, d \in \partial\Sigma$. Suppose $f$ is a path from $a$ to $b$ and $g$ is a path from $c$ to $d$ such that all intersections between $f$ and $g$ occur in $\Sigma \setminus \partial\Sigma$ at transverse double points. Suppose further that $p \neq q \in f \cap g$. Let $s, t, x, y \in [0,1]$ such that $s < t$, $x < y$, and $f(s) = g(y) = p$ and $f(t) = g(x) = q$. Then,
    \[(g_pf)_q \simeq (f_qg)_p, \quad f_pg \simeq (f_qg)_{\setminus p}, \quad \mathrm{and~} (g_pf)_{\setminus q} \simeq g_qf\]
    and
    \[\epsilon(p; f, g)\epsilon(q; g_pf) = -\epsilon(q; f, g)\epsilon(p; f_qg)\]
  \end{lem}

  \begin{proof}
    First, we prove that $q$ is in fact a self intersection point of $g_pf$. The path $g_pf$ follows $g$ from $c$ to $p$. Since $x < y$, we have that $g$ hits $q$ on this segment. Then, the path $g_pf$ follows $f$ from $p$ to $b$. Since $s < t$, we have that $f$ hits $q$ on this segment. Thus, the path $g_pf$ equals $q$ at two different points of $[0,1]$, which proves $q \in I(g_pf)$. The fact that $p \in I(f_qg)$ follows from a similar proof.

The loop $(g_pf)_q$ follows $g$ from $q$ to $p$ then follows $f$ from $p$ to $q$. The loop $(f_qg)_p$, traces the same loop but starting at $p$, so these loops are homotopic. 

The path $f_pg$ follows $f$ from $a$ to $p$ then follows $g$ from $p$ to $d$. Notably, it does not hit $q$ at all. The path $f_qg$ follows $f$ from $a$ to $q$ (passing $p$ along the way), then follows $g$ from $q$ to $d$ and hits $p$ again on the way. Thus the path $(f_qg)_{\setminus p}$ follows $f$ from $a$ to $p$, then follows $g$ from $p$ to $d$. Thus, $f_pg \simeq (f_qg)_{\setminus p}$. 

The same argument but switching the roles of $f$ with $g$ and the roles of $p$ with $q$ shows $(g_pf)_{\setminus q} \simeq g_qf$.

    Now we prove the last equality. Suppose that $u < v \in [0,1]$ such that $g_pf(u) = g_pf(v) = q$. The value of $\epsilon(q; g_pf)$ is determined by the ordering of the vectors $\{d(g_pf)_u(1), d(g_pf)_v(1)\}$. However, by what we just showed above, $d(g_pf)_u = dg_x$ and $d(g_pf)_v = df_t$, and so $\{d(g_pf)_u(1), d(g_pf)_v(1)\} = \{dg_x(1), df_t(1)\}$, and taking the orientation of this basis gives exactly the value $\epsilon(q; g, f)$. 

    Similarly, $\epsilon(p; f_qg) = \epsilon(p; f, g)$. Thus, the second statement in the proof boils down to 
    \[\epsilon(p; f, g)\epsilon(q; g, f) = -\epsilon(q; f, g)\epsilon(p; f, g)\]
    which is true because $\epsilon(q; g, f) = -\epsilon(q; f, g)$. 
  \end{proof}

  We now establish the compatibility between $\sigma$ and $[[-,-]]$. First note

  \begin{align*}
    f \otimes g \xmapsto{[[-,-]]} &\sum_{p \in f \cap g} \epsilon(p; f,g) f_pg \otimes g_pf \\
    \xmapsto{\widetilde{\sigma}} &\sum_{p \in f \cap g} \sum_{q \in I(f_pg)} \epsilon(p; f, g) \epsilon(q; f_pg) (f_pg)_q \otimes (f_gp)_{\setminus q} \otimes g_pf\\
    + &\sum_{p \in f \cap g}\sum_{q \in I(g_pf)} \epsilon(p; f, g) \epsilon(q; g_pf) (g_pf)_q \otimes f_gp \otimes (g_pf)_{\setminus q}
  \end{align*}

  and 

  \begin{align*}
    f \otimes g \xmapsto{\widetilde{\sigma}} & \sum_{q \in I(f)} \epsilon(q; f) f_q \otimes f_{\setminus q} \otimes g\\
    + &\sum_{q \in I(g)} \epsilon(q; g) g_q \otimes f \otimes g_{\setminus q}\\
    \xmapsto{1 \otimes [[-,-]]} &\sum_{q \in I(f)} \sum_{p \in f_{\setminus q} \cap g} \epsilon(q; f) \epsilon(p; f_{\setminus q}, g) f_q \otimes (f_{\setminus q})_pg \otimes g_p(f_{\setminus q})\\
    + &\sum_{q \in I(g)} \sum_{p \in f \cap g_{\setminus q}} \epsilon(q; g) \epsilon(p; f, g_{\setminus q}) g_q \otimes f_p(g_{\setminus q}) \otimes (g_{\setminus q})_pf.
  \end{align*}

  We can equate these sums, considered as elements in $L^\circ \otimes P\otimes P$, term by term using the previous three lemmas. First, the term where the intersection points $u \in f \cap g$ and $v \in I(f)$ where $f(s) = f(t) = v$, $f(r) = u$, and $r < s$ are resolved occurs in the first sum once and the second sum once. These two terms are 
  \[\epsilon(u, f, g) \epsilon(v, f_ug) (f_ug)_v \otimes (f_ug)_{\setminus v} \otimes g_uf \mathrm{~and~} \epsilon(v; f) \epsilon(u, f_{\setminus v}, g) f_v \otimes (f_{\setminus v})_ug \otimes g_u(f_{\setminus v})\]
  These two terms are equal by Lemmas \ref{twopathself} and \ref{selfsamesign}. 

  We can also use the same lemmas for the case when $t < r$. 

We now analyze the case $s < r < t$. In the second sum, after applying the coaction, $r$ is now an intersection point between $g$ and $f_v$, and $u \not\in f_{\setminus v} \cap g$, and thus no term appears in the second sum with $u$ and $v$ as the two resolved intersection points. 
  
  Similarly, in the first sum, after resolving the intersection of $f$ and $g$ at $u$, the intersection at $v$ is no longer a self intersection, it is an intersection of $f_ug$ and $g_uf$. Thus, $v \not\in I(f_ug)$ and $v \not\in I(g_uf)$, and so the two sides remain equal. 

  All of the above holds true when switching the roles of $f$ and $g$, and this covers every term in the second sum. However, it does not cover every term in the first sum: Let $u, v \in f \cap g$ and $s, t, x, y \in [0,1]$ such that $s < t$ and $x < y$. Then, if $f(s) = g(y) = u$ and $f(t) = g(x) = v$, we proved in Lemma \ref{twopathintersectcancel} that after resolving the intersection at $u$, the intersection at $v$ is a self intersection of $g_uf$. Similarly, we prove that after resolving the intersection at $v$, the intersection at $u$ is a self intersection of $g_uf$. Thus, the following two terms appear in the first sum: 
  \[\epsilon(u, f, g)\epsilon(v; g_pf)(g_uf)_v \otimes f_ug \otimes (g_uf)_{\setminus v} \mathrm{~and~} \epsilon(v; f, g) \epsilon(u; f_vg) (f_vg)_u \otimes (f_vg)_{\setminus u} \otimes g_vf\]
  However, by Lemma \ref{twopathintersectcancel}, these two terms are the same tensor with opposite sign, and so they cancel out. This matches every term in the first sum and proves the desired compatibility of $[[,]]$ with $\sigma$.

  Finally, we establish the three lemmas that will be used to prove condition (6) in \ref{doubleliebimodule}.  

  \begin{lem}\label{db2intersectcancel}
    Let $a, b, c, d \in \bd\Sigma$ and suppose $f$ is a path from $a$ to $b$ and $g$ is a path from $c$ to $d$ that intersects $f$ only at transverse double points in $\Sigma \setminus \bd\Sigma$. Let $p \neq q \in f \cap g$ with $s, t, x, y \in (0,1)$ such that $s < t, x < y$, and $f(s) = g(x) = p$ and $f(t) = g(y) = q$. Then the following hold:
    \begin{enumerate}
      \item $(f_pg)_q(g_pf) \simeq(f_qg)_p(g_qf)$ and $(g_pf)_q(f_pg) \simeq (g_qf)_p(f_qg)$
      \item $\epsilon(p; f, g)\epsilon(q; f_pg, g_pf) = -\epsilon(q; f, g)\epsilon(p; f_qg, g_qf)$. 
    \end{enumerate}
  \end{lem}

  \begin{proof}
    We first prove (1). The path $f_pg$ follows $f$ from $a$ to $p$, then follows $g$ from $p$ to $d$. Notably, since $x < y$ and $s < t$, we have the path $f_pg$ going through the point $q$ exactly once, following the path $g$. 

    The path $g_pf$ follows $g$ from $c$ to $p$, then follows $f$ from $p$ to $b$. Note again that since $x < y$ and $s < t$, the path $g_pf$ passes through $q$ exactly once, but this path follows $f$. 

    Thus, $(f_pg)_q(g_pf)$ follows $f$ from $a$ to $p$, follows $g$ from $p$ to $q$, then follows $f$ from $q$ to $b$. 

    Now, the path $f_qg$ follows $f$ from $a$, passing through $p$, to $q$, then follows $g$ from $q$ to $d$, not hitting $p$ at all. The path $g_qf$ follows $g$ from $c$, passing through $p$, to $q$, then follows $f$ from $q$ to $b$. Thus, the path $(f_qg)_p(g_qf)$ follows $f$ from $a$ to $p$, follows $g$ from $p$ to $q$, then follows $f$ from $q$ to $b$, proving that indeed 
 \[(f_pg)_q(g_pf) \simeq (f_qg)_p(g_qf).\]
    The second identity follows from a similar argument.

    The proof of (2) follows once it is observed that $q$ is in the image of the $g$ portion of $f_pg$ and the $f$ portion of $g_pf$, so that 
    \[\epsilon(q; f_pg, g_pf) = \epsilon(q; g, f)\]
    and that $p$ is in the image of the $f$ portion of $f_qg$ and the $g$ portion of $g_qf$, and so 
    \[\epsilon(p; f_qg, g_qf) = \epsilon(p; f, g).\]
  \end{proof}

  \begin{lem}\label{db2intersectoutoforder}
    Let $f, g, p, q$ be as in the previous lemma, but now suppose that $s, t, x, y \in (0,1)$ satisfy $s < t, x < y$, $f(s) = g(y) = p,$ and $f(t) = g(x) = q$. Then, $q \not\in f_pg \cap g_pf$ and $p \not\in f_qg \cap g_qf$. 
  \end{lem}

  \begin{proof}
    Note the subtle difference in the setup of this lemma as opposed to Lemma \ref{db2intersectcancel}. Now, $f$ passes through the intersection points $p, q$ in that order, but $g$ passes through them in the order $q, p$. Thus, $g_pf$ has $q$ as a self intersection, and $f_pg$ doesn't pass through $q$ at all. Similarly, $f_qg$ has $p$ as a self intersection, but $g_qf$ does not pass through $p$ at all.  
  \end{proof}

  The last lemma will consider the case when, after resolving an intersection point between two paths, a self intersection of one of the old paths is now an intersection point of one of the new paths. 

  \begin{lem}\label{dbselfintersect}
    Let $f, g$ be as in the previous lemma and let $p \in f \cap g$ be a transverse double point. 
    \begin{enumerate}
      \item Suppose $q \in I(f)$, and suppose further that $s < r < t \in (0,1)$ such that $f(s) = f(t) = q$ and $f(r) = p$. Then, $q \in f_pg \cap g_pf$ and 
      \[(f_pg)_q(g_pf) \simeq f_{\setminus q}, \quad (g_pf)_q(f_pg) \simeq (f_q)_pg\]
      and 
      \[\epsilon(p; f, g)\epsilon(q; f_pg, g_pf) = \epsilon(q; f)\epsilon(p; f_q, g).\]
      \item Suppose $q \in I(g)$, and suppose further that $s < r < t \in (0,1)$ such that $g(s) = g(t) = q$ and $g(r) = p$. Then, $q \in f_pg \cap g_pf$ and
      \[(f_pg)_q(g_pf) \simeq (g_q)_pf, \quad (g_pf)_q(f_pg) \simeq g_{\setminus q}\]
      and 
      \[\epsilon(p; f, g)\epsilon(q; f_pg, g_pf) = \epsilon(q; g)\epsilon(p; g_q, f).\]
    \end{enumerate}
  \end{lem}

  \begin{proof}
    We prove the first assertion; the second follows from a similar argument. 

    Since $s < r$, the path $f_pg$ follows $f$ from $a$, passing through $q$ once, then continuing to $p$. Then it follows $g$ from $p$ to $d$. 
    
    The path $g_pf$ follows $g$ from $c$ to $p$, then, since $r < t$, it follows $f$ from $p$, passing through $q$ once, and then continuing to $b$.

    Thus, the path $(f_pg)_q(g_pf)$ follows $f$ until $q$, then jumps to the path $g_pf$, which means it follows $f$ from $q$ to $b$, skipping the image of $f|_{[s, t]}$. Thus, $(f_pg)_q(g_pf) \simeq f_{\setminus q}$. 

    The path $(f_q)_pg$ starts at $c$, follows $g$ until $p$, then follows the loop $f_q$ all the way around, then follows $g$ from $p$ to $d$. 
    
    The path $(g_pf)_q(f_pg)$ follows $g$ from $c$ to $p$, then follows $f$ until $q$, where it jumps to $f_pg$, and thus follows $f$ again until $p$, and then follows $g$ from $p$ to $d$. This is readily seen to be the same as the description given above for $(f_q)_pg$, up to reparameterization.

    As for the signs, it is clear that $\epsilon(p; f, g) = \epsilon(p; f_q, g)$. The sign of $\epsilon(q; f)$ is given by the orientation of $\{df_s(1), df_t(1)\}$. The path $f_pg$ is the path $f|_{[0, r]}$ concatenated with the appropriate restriction of $g$. Since $s \in [0,r]$, the unit vector tangent to $f_pg$ at $q$ is the same as $df_s(1)$. Similarly, the unit vector tangent to $g_pf$ at $q$ is the same as $df_t(1)$. Thus, $\epsilon(q; f) = \epsilon(q; f_pg, g_pf)$. 
  \end{proof}

  We prove condition (6) in \ref{doubleliebimodule}. Let $a, b, c, d \in \bd\Sigma$. Let $f$ be a path from $a$ to $b$ and let $g$ be a path from $c$ to $d$ that intersects $f$ at only transverse double points. We want to establish the following identity in $P \otimes P$:
  \[\big[\big[[[f, g]]^1, [[f, g]]^2\big]\big] = \rho(\sigma(g)^1, f) \otimes \sigma(g)^2 + \ol\sigma(f)^1 \otimes \ol\rho(g, \ol\sigma(f)^2).\]
  We write out the left hand side: 
  \begin{align*}
    [[f, g]]^1 \otimes [[f,g]]^2 =& \sum_{p \in f \cap g} \epsilon(p; f, g) f_pg \otimes g_pf\\
    \xmapsto{[[-,-]]}& \sum_{p \in f \cap g} \sum_{q \in f_pg \cap g_pf} \epsilon(p; f, g) \epsilon(q; f_pg, g_pf) (f_pg)_q(g_pf) \otimes (g_pf)_q(f_pg)
  \end{align*}
  By Lemmas \ref{db2intersectcancel} and \ref{db2intersectoutoforder}, any terms where $q \in f \cap g$ are zero (either because the term cancels out with a corresponding term where the roles of $p$ and $q$ are swapped, or because $q \not\in f_pg \cap g_pf$). Thus, the only nonzero terms are when $q \in I(g)$ or $q \in I(f)$, which are exactly the two scenarios in Lemma \ref{dbselfintersect}. 

 We now write out the two terms in the right hand side:
  \begin{align*}
    \rho(\sigma(g)^1, f) \otimes \sigma(g)^2 =& \sum_{q \in I(g)} \epsilon(q; g) \rho(g_{q}, f) \otimes g_{\setminus q} \\
    =& \sum_{q \in I(g)} \sum_{p \in g_q \cap f} \epsilon(q; g) \epsilon(p; g_q, f) (g_q)_pf \otimes g_{\setminus q}\\
    \ol\sigma(f)^1 \otimes \ol\rho(g, \ol\sigma(f)^2) =& \sum_{q \in I(f)} -\epsilon(q; f) f_{\setminus q} \otimes \ol\rho(g, f_q)\\
    =& \sum_{q \in I(f)}\sum_{p \in f_q \cap g} \epsilon(q; f) \epsilon(p; f_q, g) f_{\setminus q} \otimes (f_q)_pg
  \end{align*}
  The only terms that appear in the above sum are when $q \in I(g)$ and $p \in g_q \cap f$, and thus $q \in g_pf \cap f_pg$ and when $q \in I(f)$ and $p \in g \cap f_q$, and so $q \in g_pf \cap f_pg$. It then follows by Lemma \ref{dbselfintersect} that these two sums are the same. 

  Thus, $(L^\circ,[-,-],\delta, P, \rho,\sigma, [[-,-]])$ is an involutive double Lie $L^\circ$-bimodule. 
\qed

\begin{rmk} Yusuke Kuno communicated to the authors that the compatibility between $\rho$ and $[[-,-]]$ (namely, the double Lie-module compatibility) was stated without proof in formula (5.4) of Section 5.2 in \cite{KK16} and that a proof to a formula equivalent to it was given in Lemma 7.4 of \cite{MT13}.
\end{rmk}

\begin{cor}
  The map $\{-,-\} : L \oplus Cyc(P) \otimes L \oplus Cyc(P) \to L \oplus Cyc(P)$ defined as in \ref{liebracket} using the Goldman bracket, the module actions $\rho$ and $\ol\rho$, and the compatible double bracket $[[,]]$ is a Lie bracket. 
\end{cor}

\begin{proof}
  This follows directly from Theorem \ref{liebracket}. 
\end{proof}

\begin{cor}
  The map $\Delta: L \oplus Cyc(P) \to (L \oplus Cyc(P)) \otimes (L \oplus Cyc(P))$ defined as in \ref{liecobracket} using the Turaev cobracket, the comodule coactions $\sigma$ and $\ol\sigma$, and the cocompatible double bracket $[[,]]$ is a Lie cobracket. 
\end{cor}

\begin{proof}
  This follows directly from Theorem \ref{liecobracket}
\end{proof}

\begin{cor}
  The triple $(L \oplus Cyc(P), \{-,-\}, \Delta)$ is a Lie bialgebra.
\end{cor}

\begin{proof}
  This follows directly from Theorem \ref{liebialgebra}.
\end{proof}


\section{Splitting the Goldman-Turaev Lie Bialgebra}
\label{Sec:DeconstructingGTLB}
Let $\Sigma$ be an oriented surface and $\alpha$ a simple separating curve on $\Sigma$. By simple we mean that $\alpha$ has no self-intersections, and by separating we mean that if the surface $\Sigma$ is cut along the image of $\alpha$, this results in two new oriented surfaces $\Sigma_1$ and $\Sigma_2$ such that 
\begin{enumerate}
  \item $\Sigma_1$ and $\Sigma_2$ each have the image of $\alpha$ as a boundary component, and 
  \item $\Sigma_1 \sqcup_{\alpha} \Sigma_2 = \Sigma$. 
\end{enumerate}
Throughout this section, we will use the following conventions: 
\begin{enumerate}
  \item If $\gamma$ is a loop on $\Sigma$ or $\Sigma_i$, we will denote both the loop and its homotopy class by $\gamma$. 
  \item If $a \neq b \in \partial\Sigma_i$ and $f$ is a path from $a$ to $b$, we will also denote the class of paths homotopic to $f$ through paths preserving the endpoints as $f$. 
  \item Where the distinction between a path $\beta$ and its image are clear, we will represent its image also by $\beta$. 
\end{enumerate}

Let $L^\circ(\Sigma)$ and $L_i^\circ$ denote the GT Lie bialgebras of $\Sigma$ and $\Sigma_i$, respectively, for $i=1,2$. Take $I = \alpha$ as the set of objects for the linear quiver with $P_i(a,b)$ being the linear span of the set of homotopy classes of paths in $\Sigma_i$ that start at $a$ and end at $b$, where $a,b \in \alpha$. Denote $P_i=\bigoplus_{(a,b) \in \alpha \times \alpha} P_i(a,b)$ for $i=1,2$. Consider the involutive double Lie bimodule  $(L_i^\circ,[-,-]_i,\delta_i, P_i, \rho_i,\sigma_i, [[-,-]]_i)$, for $i=1,2$, given by applying Theorem \ref{doubleLiebimodulethm} to $\Sigma_i$ with $\alpha$ as the set of objects and equip $L_1^\circ \oplus L_2^\circ \oplus (P_1 \boxtimes P_2)$ with the involutive Lie bialgebra structure given in Theorem \ref{weavingbialgebra}.

\begin{thm}\label{gluingmap}
 The map
  \[L_1^\circ \oplus L_2^\circ \oplus (P_1 \boxtimes P_2) \xrightarrow{\varphi} L^\circ(\Sigma)\]
  given by  
  \[\varphi(\alpha; \beta; (f_1, g_2, f_3, g_4 \dots, f_{n-1}, g_n)) = \iota_1(\alpha) + \iota_2(\beta) + \iota_1(f_1) * \iota_2(g_2) * \cdots * \iota_1(f_{n-1}) * \iota_2(g_n), \]
  where $\iota_i : \Sigma_i \to \Sigma$ is the the inclusion map and $*$ denotes concatenation of paths, is a surjective Lie bialgebra morphism.
\end{thm}

\begin{proof}
  Throughout this proof we will drop the maps $\iota_i$ from the notation when it does not cause confusion. 

We first show $\varphi$ is surjective. Let $\gamma$ be a loop in $\Sigma$. If $\gamma$ is homotopic to a loop $\gamma'$ whose image lies entirely in $\Sigma_1$, then we have $\varphi(\gamma'; 0; 0) = \gamma$. Similarly, if $\gamma'$ lies entirely in $\Sigma_2$, then $\varphi(0; \gamma'; 0) = \gamma$. 

  Now suppose that $\gamma$ does not have a representative lying entirely in $\Sigma_1$ or $\Sigma_2$. Then, $\gamma$ must intersect $\alpha$ at least twice, possibly more times. If necessary, pick a different representative that intersects $\alpha$ only doubly and transversely and denote that new representative by $\gamma$. 
  
  Pick any point $p_1 \in \gamma \cap \alpha$. Let $f_1 : [0,1] \to \Sigma$ be the arc of $\gamma$ that starts at $p_1$ and ends at the next intersection of $\gamma$ with $\alpha$ (we know $\gamma$ must intersect $\alpha$ again because $\alpha$ is a separating curve). Call that intersection point $p_2$. Now let $g_1 : [0,1] \to \Sigma$ be the arc of $\gamma$ that starts at $p_2$ and ends at the next intersection point of $\gamma$ with $\alpha$. If that point is $p_1$, we write $\gamma = f_1 * g_1$ and so $\varphi(0;0; f_1, g_1) = \gamma$. Otherwise, we label this new point $p_3$. This process continues until eventually we have $\gamma = f_1 * g_1 * f_2 * g_2 * \cdots * f_n * g_n$, and thus $\varphi(0;0; f_1, \dots, g_n) = \gamma$. 

 We now check $\varphi$ is a morphism of Lie algebras. Let $F, F' \in V = L_1^\circ \oplus L_2^\circ \oplus (P_1 \boxtimes P_2)$. Since $V$ is a direct sum, we need only consider the six cases where $F$ and $F'$ are contained only in each one of the summands:

  Case 1: $F = (\alpha; 0; 0)$ and $F' (\alpha'; 0; 0)$ and Case 2: $F = (0; \beta; 0)$ and $F' = (0; \beta'; 0)$. Since $\varphi$ is the inclusion map on $L_1^\circ$ and $L_2^\circ$, this case is obvious. 

  Case 3: $F = (\alpha; 0; 0)$ and $F' = (0; \beta'; 0)$: We fist see that $\{(\alpha; 0; 0), (0; \beta'; 0)\} = 0$. Since the loop $\iota_1(\alpha)$ does not intersect $\iota_2(\beta')$ we also have $[\varphi(\alpha; 0; 0), \varphi(0; \beta'; 0)] = 0$. 

  Case 4: $F = (\alpha; 0; 0)$ and $F' = (0; 0; f_1, g_1, \dots, f_n, g_n)$. Each term of the bracket $\{(\alpha; 0; 0),\allowbreak (0; 0; f_1, g_1, \dots, f_n, g_n)\}$ is given by resolving one intersection point of $\alpha$ with some $f_i$ for $i \in \{1, \dots, n\}$. Fix a single term in the bracket and write it as 
  \[\epsilon(p; \alpha, f_i) f_1, g_1, \dots, \alpha_pf_i, g_i, \dots, f_n, g_n\]
  This maps by $\varphi$ to
  \[\epsilon(p; \alpha, f_i) f_1 * g_1 * \cdots * \alpha_pf_i * g_i * \cdots * f_n * g_n,\]
  which is the same as 
  \[\epsilon(p; \alpha, f_i) \alpha_p(f_1 * g_1 * \cdots * f_i * g_i * \cdots * f_n * g_n),\]
  which is the term associated to $p$ in the bracket $[\varphi(\alpha; 0; 0), \varphi(0; 0; f_1, \dots, g_n)]$. Since $p$ was arbitrary, all the terms on each side match up. 

  Case 5: $F = (0; \beta; 0)$ and $F' = (0; 0; f_1, g_1, \dots, f_n, g_n)$: Proven analogously to Case 4. 

  Case 6: $F = (0; 0; f_1, g_1, \dots, f_n, g_n)$ and $F' = (0; 0; f_1', g_1', \dots, f_m', g_m')$. Once again, we pick out an individual term in $[F, F']$ corresponding (without loss of generality) to the intersection point $p \in f_i \cap f_j'$. The term of $\{F, F'\}$ corresponding to $p$ is 
  \[\epsilon(p; f_i, f_j') (0; 0; ({f_i})_pf_j', g_j', f_{j+1}', g_{j+1}', \dots, f_{j-1}', g_{j-1}', ({f_j'})_pf_i, g_i, f_{i+1}, g_{i+1}, \dots, f_{i-1}, g_{i-1}).\]
  Applying $\varphi$ gives us 
  \[\epsilon(p; f_i, f_j') ({f_i})_pf_j' * g_j'* f_{j+1}'* g_{j+1}'* \cdots *f_{j-1}' *g_{j-1}'* ({f_j'})_pf_i* g_i* f_{i+1}* g_{i+1}* \cdots *f_{i-1}* g_{i-1},\]
  which is the same as 
  \[\epsilon(p; f_i, f_j') (f_j' * g_j' * f_{j+1}'* g_{j+1}' * \cdots f_{j-1}'* g_{j-1}')_p (f_i * g_i * f_{i+1}* g_{j+1}* \cdots * f_{i-1}* g_{i-1}),\]
  which is the corresponding term of $[\varphi(F), \varphi(F')]$. The same holds true for intersection points in $g_i \cap g_j'$. 

  Thus, in every case, $\varphi$ is a Lie algebra morphism. See figure \ref{fig:weaving} for an example. 

  \begin{figure}
    \centering
    \input{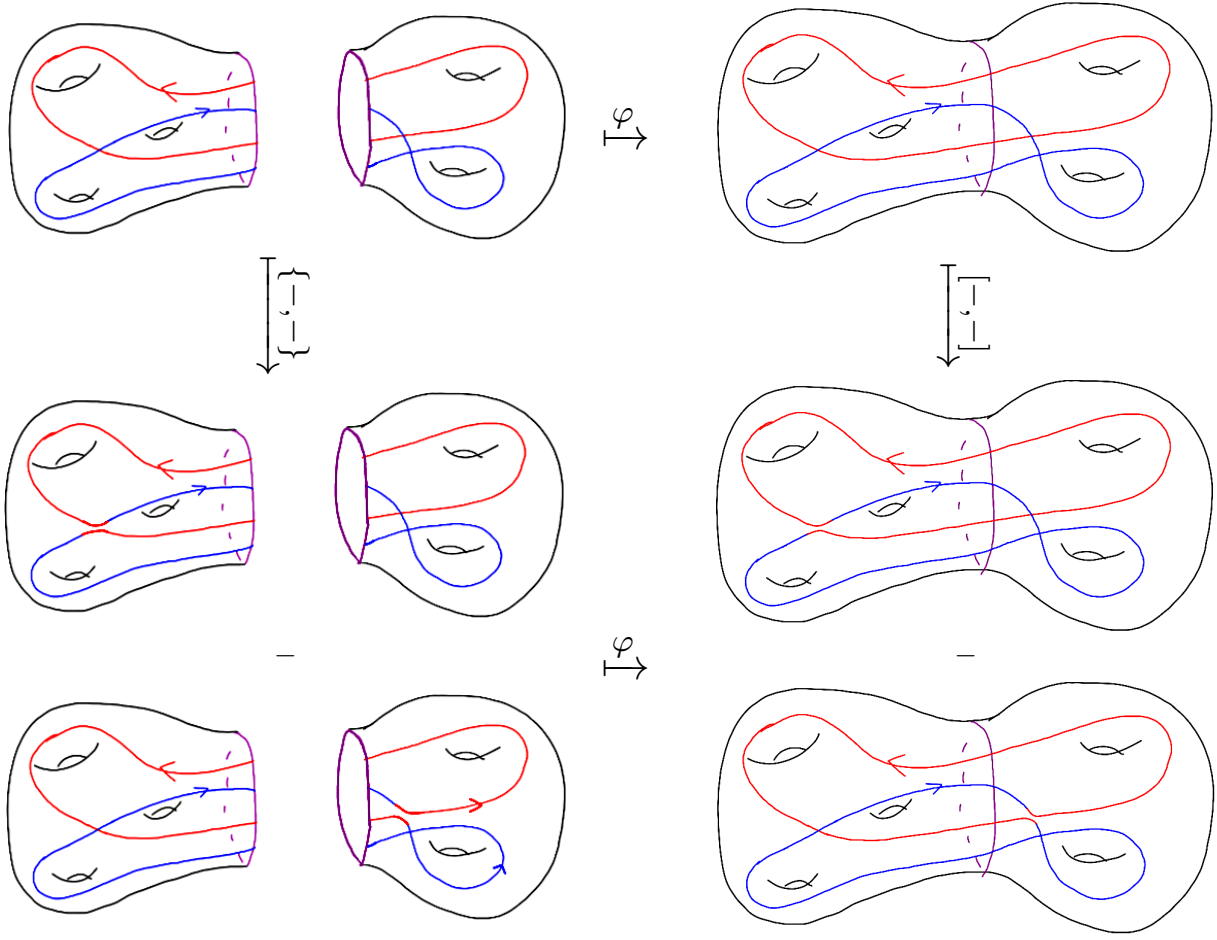}
    \caption{The map $\varphi$ is a Lie algebra morphism}
    \label{fig:weaving}
  \end{figure}

  The fact that $\varphi$ is a morphism of Lie coalgebras is proven similarly. Note that $F = (0; 0; f_1, g_1, \dots, f_n, g_n)$ in $V$ may have a nontrivial intersection in $f_i \cap f_j$, which after applying $\varphi$, creates a contractible monogon. This turns out not to be a concern. Suppose $q \in f_i \cap f_j$ is such an intersection point in $f_1* g_1 * \cdots * f_n *g_n$. After applying the cobracket $\Delta(F)$, one side of the tensor of the term corresponding to the point $q$ will map by $\varphi$ to a constant loop, and thus the whole term will map to 0. On the other hand, after taking $\varphi(F)$, we can kill the monogon and so no term corresponding to $q$ appears in $\delta(\varphi(F))$. See figure \ref{fig:monogon} for an example. 

  \clearpage

  \begin{figure}
    \centering
    \input{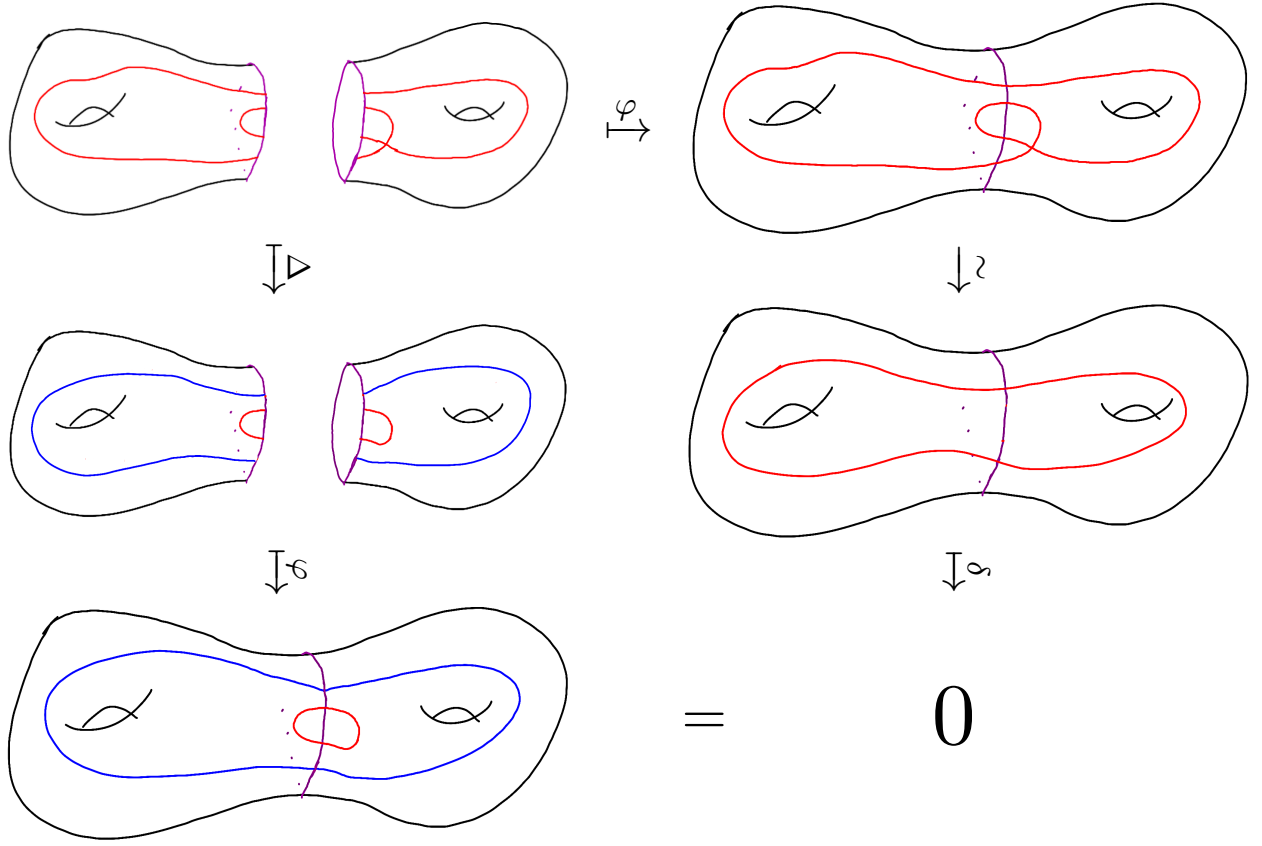}
    \caption{Dealing with potential monogons created by $\varphi$.}
    \label{fig:monogon}
  \end{figure}
\end{proof}

The kernel of $\varphi$ is generated by the following relations.

\begin{thm}\label{kerofvarphi}
If $\gamma^{x,y} : [0,1] \to \Sigma$ is a path along the boundary from $x$ to $y$ (note $x = y$ is allowed), denote by $\gamma_i^{x,y} : [0,1] \to \Sigma_i$ the path $\gamma$ but viewed as a path in $\Sigma_i$. For each $i$, let $a_i$ be the start point of either $f_i$ or $g_i$, depending on of $i$ is even or odd. Then, we have 
  \begin{align}
    \ker(\varphi) = &\big\langle(\dots, f_i, g_{i+1}, \dots) - (\dots, f_i*\gamma_1^{a_{i+1},y}, \overline{\gamma_2}^{y,a_{i+1}}*g_{i+1}, \dots),\\
    &(\dots, g_i, f_{i+1}, \dots) - (\dots, g_i*\gamma_2^{a_{i+1},y}, \overline{\gamma_1}^{y,a_{i+1}}*f_{i+1}, \dots),\\
    &(\dots, f_i, \gamma_2^{a_{i+1},a_{i+2}}, f_{i+2}, \dots) - (\dots, f_i*\gamma_1^{a_{i+1},a_{i+2}}*f_{i+2}, \dots),\\
    &(\dots, g_i, \gamma_1^{a_{i+1},a_{i+2}}, g_{i+2}, \dots) - (\dots, g_i*\gamma_2^{a_{i+1},a_{i+2}}*g_{i+2}, \dots),\\
    &(f_1, \gamma_2^{a_2,a_1}) - f_1*\gamma_1^{a_2,a_1}, \quad (\gamma_1^{a_1, a_2}, g_2) - \gamma_2^{a_2,a_1}*g_2\big\rangle
  \end{align}
  where $*$ denotes concatenation of paths and in the last two equations, $f_1*\gamma_1^{a_2,a_1}$ is taken as an element of $L_1$ and similarly $\gamma_2^{a_2,a_1}*g_2$ as an element of $L_2$. 
\end{thm}

\begin{proof}
  Let $S$ denote the ideal on the right hand side of the above equation. The fact that $S \subset \ker(\varphi)$ is obvious, so we prove $\ker(\varphi) \subset S$. Let $x \in \ker(\varphi)$. We can take $x$ to have the form 
  \[\sum_{i=1}^k p_ia_i + \sum_{i=1}^l q_ib_i + \sum_{i=1}^m r_i(f_{1}, g_{2}, \dots, f_{(n - 1)}, g_{n})_i\]
  where $a_i \not\simeq a_j, b_i \not\simeq b_j$ for $i \neq j$, and $p_i, q_i, r_i \in \Bbbk$. We can also assume that for any $i \neq j$, the two terms $(f_1, \cdots, g_{n})_i$ and $(f_{1}, \cdots, g_{n})_j$ are different in at least one position. Since  $x \in \ker(\varphi)$ we have
  \[\sum_{i=1}^k p_i\iota_1(a_i) + \sum_{i=1}^l q_i\iota_2(b_i) + \sum_{i=1}^m r_i(\iota_1(f_{1}) * \iota_2(g_{2})*\cdots*\iota_1(f_{(n-1)})*\iota_2(g_{n}))_i = 0\]
  We will now drop the $\iota$ notation and allow the reader to assume from context in which space a path or loop lives. Since each $a_i, b_i$ or concatenation of $f$'s and $g$'s is a single loop, the only way we get cancellation of terms to zero is when the coefficients match up in the following 6 cases: 
  \begin{enumerate}
    \item $a_i \simeq a_j$ for $i \neq j$
    \item $a_i \simeq b_j$ 
    \item $a_i \simeq (f_{1} * \cdots * g_{n})_j$
    \item $b_i \simeq b_j$ for $i \neq j$
    \item $b_i \simeq (f_{1} * \cdots * g_{n})_j$, or
    \item $(f_{1} * \cdots * g_{m})_i \simeq (f_{1} * \cdots * g_{n})_j$
  \end{enumerate}
  We now prove that in each of these 6 cases, the left hand side minus right hand side of each homotopy relation can be written as a sum of elements of $S$ of the forms appearing in lines (11) - (15). 
  
  Cases 1 and 4: If $a_i \simeq a_j$ in $\Sigma$, then $a_i \simeq a_j$ in $\Sigma_1$ as well. Since we assumed that each of the $a_i$'s and $b_i$'s are homotopically distinct, these cases cannot occur. Similarly, Case 4 where $b_i \simeq b_j$ cannot occur.

  Case 2: $a_i \simeq b_j$. Then, there is a homotopy $H : [0,1]^2 \to \Sigma$ such that $H(x, 0) = a_i(x)$ and $H(x,1) = b_j(x)$. Since $\alpha$ is a separating curve on $\Sigma$ and $a_i \subset \Sigma_1$ and $b_j \subset \Sigma_2$, for any $x \in [0,1]$, there is a $t_x$ such that 
  \[H(x, t_x) \in \alpha.\]
 Define
  \[H'(x, t) = \begin{cases}
    H(x, t) & t < t_x \\ H(x, t_x) & t \geq t_x,
  \end{cases}\]
and let $\gamma(x) \coloneqq H'(x, 1)$. Then, $H'$ is a homotopy between $a_i$ and $\gamma$, which is a loop lying entirely in the image of $\alpha$. Pick $x < y \in [0,1]$ such that $\gamma(x) \neq \gamma(y)$. Let $\gamma_1 : [0,1] \to \Sigma_1$ be the path $\gamma|_{[x,y]}$ and let $\gamma_2 : [0,1] \to \Sigma_2$ be the path $\gamma|_{[y,1] \sqcup [0, x]}$. Since $\gamma(0) = \gamma(1)$, this is well defined. Then, we take $\gamma_1$ as an element of $\Sigma_1$ and $\gamma_2$ as an element of $\Sigma_2$ and get
  \[a_i - b_j = (a_i - (\gamma_1,\gamma_2)) + ((\gamma_1, \gamma_2) - b_j) \in S.\]
  
  Case 3: $a_i \simeq (f_{1} * \cdots * g_{n})_j$. Then there is a homotopy $H : [0,1]^2 \to \Sigma$ such that $H(x, 0) = a_i(x)$ and $H(x, 1) = (f_{1} * \cdots * g_{n})_j(x)$. Since $a \in \Sigma_1$, for similar reasons as those in case 2, for every even $k$ between $1$ and $n$, there is a homotopy between $g_{k}$ and a path $\gamma_{k}$ on the boundary. Thus, 
  \begin{align*}
    (f_{1}, g_{2}, f_{3}, g_{4}, \dots&, f_{(n-1)}, g_{n})_j - a_i \\
    = ((f_{1}, \gamma_{2}, f_{3}, \gamma_{4}, \dots, f_{(n-1)}, \gamma_{n})_j& - (f_{1} * \gamma_{2} * f_{3}, \gamma_{4}, \dots, f_{(n-1)}, \gamma_{n})_j) \\
    + ((f_{1} * \gamma_{2} * f_{3}, \gamma_{4}, \dots, f_{(n-1)}, \gamma_{n})_j& - (f_{1} * \gamma_{2} * f_{3} * \gamma_{4} * f_{5}, \dots, f_{(n-1)}, \gamma_{n})_j) \\
    + \dots + ((f_{1} * \gamma_{2} * \cdots& * f_{(n-1)}, \gamma_{n})_j - a_i) \in S
  \end{align*}
  where the last subtraction is in $S$ because $a_i \simeq (f_{1} * \cdots * \gamma_{n})_j$ by hypothesis. 

  Case 5: $b_i \simeq (f_{1} * \cdots * g_{n})_j$. This is proven as in Case 3 switching the roles of the $f$'s and $g$'s.

  Case 6: $(f_{1} * \cdots * g_{m})_i \simeq (f_{1}' * \cdots * g_{n}')_j$. 

  Pick a point $x \in \alpha$ as the basepoint of $\Sigma, \Sigma_1,$ and $\Sigma_2$. By Lemma 5.2 of \cite{Cha04}, for each surface with boundary $\Sigma_i$, there is an alphabet $A_i$ such that the free group on letters in $A_i$ is the fundamental group $\pi_1(\Sigma_i, x)$ and the cyclically reduced cyclic words in $A_i$ represent the conjugacy classes of $\pi_1(\Sigma_i, x)$, which correspond to the free homotopy classes of loops, and the boundary loop $\alpha$ is represented by a single letter in $A_i$. For each $i \in \{1, 2\}$, we let $\alpha_i$ be the letter representing $\alpha$ in $A_i$. 

  By the Seifert-Van Kampen theorem, 
  \[\pi_1(\Sigma, x) \cong \pi_1(\Sigma_1, x) * \pi_1(\Sigma_2, x)/\langle \alpha_1 - \alpha_2 \rangle\]
  and thus homotopy classes of free loops in $\Sigma$ can be represented by cyclically reduced cyclic words in $A = A_1 \sqcup A_2$, where $\alpha_1$ can be freely replaced by $\alpha_2$ and vice versa. 
  
  For every $y \in \alpha$, let $\gamma^{y,x}$ be the positively oriented path from $y$ to $x$ along the boundary. For each $k \in \{1, \dots, n\}$, let $a_k$ be the start point of $f_{k}$ if $k$ is odd, or the start point of $g_{k}$ if $k$ is even. Then, we see that
  \[\overline{\gamma^{a_k, x}} * f_{k} * \gamma^{a_{k+1}, x}\]
  is an element of $\pi_1(\Sigma_1, x)$, and thus is represented by $q^k_1\cdots q^k_{o_k}$ in the free group on $A_1$. We also see 
  \[\overline{\gamma^{a_k, x}} * g_{k} * \gamma^{a_{k+1}, x}\]
  is an element of $\pi_1(\Sigma_2, x)$ and thus is represented by $p^k_{1}\cdots p^k_{o_k}$ in the free group on $A_2$. 

  We also see that 
  \[f_{1} * \cdots * g_{n} \simeq \overline{\gamma^{a_1, x}} * f_{1} * \gamma^{a_2, x} * \cdots * \overline{\gamma^{a_n, x}} * g_{n} * \gamma^{a_1, x},\]
  which shows that the loop $f_{1} * \cdots * g_{m}$ is represented in $A$ by the cyclically reduced cyclic word 
  \[V = q^1_{1}\cdots q^1_{o_{1}}p^2_{1}\cdots p^2_{o_2}\cdots q^{(m-1)}_{1}\cdots q^{(m-1)}_{o_{(m-1)}}p^m_{1}\cdots p^m_{o_{m}}\]
  
Similarly, the loop $f_{1}'*\cdots *g_{n}'$ is represented in $A$ by 
  \[W = {q^{1}_{1}}'\cdots {q^1_{v_{1}}}'{p^2_{1}}'\cdots {p^2_{v_{2}}}'\cdots {q^{(n-1)}_{1}}'\cdots {q^{(n-1)}_{v_{(n-1)}}}' {p^n_{1}}'\cdots {p^n_{v_{n}}}'\]
  where ${q^k_{w}}' \in A_1, {p^k_{w}}' \in A_2$. 

  Because these are reduced cyclic words representing two homotopic loops, we have $V = W$, and furthermore, there is a letter $s$ in the second word which is equal to $q^1_{1}$. Cycle the second word so that the letter $s$ is in the first slot. Now, each letter in the first word matches exactly with the letter in the same position in the second word. 

  In any position $z$, there are only 4 possible cases. Either the letters in position $z$ in $V$ and $W$ are both in $A_1$, both in $A_2$, or there is one in each. If the letters in position $z$ are both in $A_i$ for $i \in \{1, 2\}$, then by the equality of $V$ and $W$, they are the same letter. 
  
 Suppose now that in position $z$ there is a $q^s_{r}$ in $V$ and a ${p^u_{t}}'$ in $W$. By construction, $q^s_{r} \in A_1$ and ${p^u_{t}}' \in A_2$. However, we must have $q^s_{r} = {p^u_{t}}'$, which means that $q^s_{r} = \alpha_1$ and ${p^u_{t}}' = \alpha_2$. A similar fact is true if instead $p^u_{t} \in V$ and ${q^s_{r}}' \in W$. Thus, the only possible differences between $(f_{1}, \dots, g_{m})_i$ and $(f_{1}, \dots, g_{n})_j$ happen when a path or a segment of a path is homotopic to a curve along the boundary or that the point on $\alpha$ where $f_{k}$ is glued to $g_{(k+1)}$ is different than where the corresponding $f_{l}'$ is glued to $g_{(l+1)}'$. Each of these differences is captured by successively applying equations (2) - (5) in $S$. Thus, $(f_{1} * \cdots * g_{m})_i \simeq (f_{1} * \cdots * g_{n})_j$ implies $(f_{1}, \dots, g_{n})_i - (f_{1}, \dots, g_{m})_j \in S$. 

\end{proof}

\section{Appendix}
\label{Sec:Appendix}
Throughout this section, we drop parentheses and zeroes in the tuple notation for elements in a direct sum.

\begin{proof}[Proof of Theorem \ref{liebracket}]
  Skew symmetry of $\{-,-\}$ is obvious. We write out the proof of the Jacobi identity in full detail below. Let $a = (\alpha; f_1, \dots, f_n), b = (\beta; g_1, \dots, g_m)$, $c = (\gamma, h_1, \dots, h_k) \in W= L \oplus Cyc(P)$. 
  {\allowdisplaybreaks
  \setcounter{equation}{0}
  \begin{align}
    \{a, \{b, c\}\} &= [\alpha, [\beta, \gamma]]\\
    &+\sum_{i=1}^n f_1, \dots, \ol\rho(f_i, [\beta, \gamma]), \dots, f_n\\
    &+ \sum_{l=1}^k h_1, \dots, \rho(\alpha, \rho(\beta, h_l)), \dots, h_k\\
    &+\sum_{l=1}^k\sum_{l'\neq l} h_1, \dots, \rho(\alpha, h_{l'}), \dots, \rho(\beta, h_l), \dots, h_k\\
    &+\sum_{l=1}^k \sum_{i=1}^n [[f_i, \rho(\beta, h_l)]]^1, h_{l+1}, \dots, h_{l-1}, [[f_i, \rho(\beta, h_l)]]^2, f_{i+1}, \dots, f_{i-1}\\
    &+\sum_{l=1}^k \sum_{i=1}^n \sum_{l' \neq l} [[f_i, h_{l'}]]^1, h_{l'+1}, \dots, \rho(\beta, h_l), \dots ,h_{l'-1}, [[f_i, h_{l'}]]^2, f_{i+1}, \dots, f_{i-1}\\
    &+\sum_{j=1}^m g_1, \dots, \rho(\alpha, \ol\rho(g_j, \gamma)), \dots, g_m\\
    &+ \sum_{j=1}^m \sum_{j' \neq j} g_1, \dots, \rho(\alpha, g_{j'}), \dots, \ol\rho(g_j, \gamma), \dots, g_m\\
    &+ \sum_{j=1}^m \sum_{i=1}^n [[f_i, \ol\rho(g_j, \gamma)]]^1, g_{j+1}, \dots, g_{j-1}, [[f_i, \ol\rho(g_j, \gamma)]]^2, f_{i+1}, \dots, f_{i-1}\\
    &+\sum_{j=1}^m \sum_{i=1}^n \sum_{j'\neq j} [[f_i, g_{j'}]]^1, g_{j'+1}, \dots, \ol\rho(g_j, \gamma), \dots, g_{j'-1}, [[f_i, g_{j'}]]^2, f_{i+1}, \dots, f_{i-1}\\
    &+\sum_{j=1}^m\sum_{l=1}^k \rho(\alpha, [[g_j, h_l]]^1), h_{l+1}, \dots, h_{l-1}, [[g_j, h_l]]^2, g_{j+1}, \dots, g_{j-1}\\
    &+\sum_{j=1}^m\sum_{l=1}^k [[g_j, h_l]]^1, h_{l+1}, \dots, h_{l-1}, \rho(\alpha, [[g_j, h_l]]^2), g_{j+1}, \dots, g_{j-1}\\
    &+\sum_{j=1}^m\sum_{l=1}^k\sum_{j'\neq j} [[g_j, h_l]]^1, h_{l+1}, \dots, h_{l-1}, [[g_j, h_l]]^2, g_{j+1}, \dots, \rho(\alpha, g_{j'}), \dots, g_{j-1}\\
    &+\sum_{j=1}^m\sum_{l=1}^k\sum_{l'\neq l} [[g_j, h_l]]^1, h_{l+1}, \dots, \rho(\alpha, h_{l'}), \dots, h_{l-1}, [[g_j, h_l]]^2, g_{j+1}, \dots, g_{j-1}\\
    &+\sum_{j=1}^m \sum_{l=1}^k \sum_{i=1}^n [[f_i, [[g_j, h_l]]^1]]^1, h_{l+1}, \dots, [[g_j, h_l]]^2, g_{j+1}, \dots, [[f_i, [[g_j, h_l]]^1]]^2, f_{i+1}, \dots, f_{i-1}\\
    &+\sum_{j=1}^m \sum_{l=1}^k \sum_{i=1}^n [[f_i, [[g_j, h_l]]^2]]^1, g_{j+1}, \dots, [[g_j, h_l]]^1, h_{l+1}, \dots, [[f_i, [[g_j, h_l]]^2]]^2, f_{i+1}, \dots, f_{i-1}\\
    &+\sum_{j=1}^m \sum_{l=1}^k \sum_{i=1}^n \sum_{j'\neq j} [[f_i, g_{j'}]]^1, g_{j'+1}, \dots, [[g_j, h_l]]^1, h_{l+1}, \dots, [[g_j, h_l]]^2, g_{j+1}, \dots, g_{j'-1},\\&\qquad\qquad\qquad\qquad\qquad[[f_i, g_{j'}]]^2, f_{i+1}, \dots, f_{i-1}\notag\\
    &+\sum_{j=1}^m \sum_{l=1}^k \sum_{i=1}^n \sum_{l'\neq l} [[f_i, h_{l'}]]^1, h_{l'+1}, \dots, [[g_j, h_l]]^2, g_{j+1}, \dots, [[g_j, h_l]]^1, h_{l+1}, \dots, h_{l'-1},\\&\qquad\qquad\qquad\qquad\qquad[[f_i, h_{l'}]]^2, f_{i+1}, \dots, f_{i-1}\notag\\
    \{b, \{c, a\}\} &= [\beta, [\gamma, \alpha]]\\
    &+\sum_{j=1}^m g_1, \dots, \overline{\rho}(g_j, [\gamma, \alpha]), \dots, g_m\\
    &+ \sum_{i=1}^n f_1, \dots, \rho(\beta, \rho(\gamma, f_i)), \dots, f_n\\
    &+\sum_{i=1}^n\sum_{i'\neq i} f_1, \dots, \rho(\beta, f_{i'}), \dots, \rho(\gamma, f_i), \dots, f_n\\
    &+\sum_{i=1}^n \sum_{j=1}^m [[g_j, \rho(\gamma, f_i)]]^1, f_{i+1}, \dots, f_{i-1}, [[g_j, \rho(\gamma, f_i)]]^2, g_{j+1}, \dots, g_{j-1}\\
    &+\sum_{i=1}^n \sum_{j=1}^m \sum_{i' \neq i} [[g_j, f_{i'}]]^1, f_{i'+1}, \dots, \rho(\gamma, f_i), \dots ,f_{i'-1}, [[g_j, f_{i'}]]^2, g_{j+1}, \dots, g_{j-1}\\
    &+\sum_{l=1}^k h_1, \dots, \rho(\beta, \overline{\rho}(h_l, \alpha)), \dots, h_k\\
    &+ \sum_{l=1}^k \sum_{l' \neq l} h_1, \dots, \rho(\beta, h_{l'}), \dots, \overline{\rho}(h_l, \alpha), \dots, h_k\\
    &+ \sum_{l=1}^k \sum_{j=1}^m [[g_j, \overline{\rho}(h_l, \alpha)]]^1, h_{l+1}, \dots, h_{l-1}, [[g_j, \overline{\rho}(h_l, \alpha)]]^2, g_{j+1}, \dots, g_{j-1}\\
    &+\sum_{l=1}^k \sum_{j=1}^m \sum_{l'\neq l} [[g_j, h_{l'}]]^1, h_{l'+1}, \dots, \overline{\rho}(h_l, \alpha), \dots, h_{l'-1}, [[g_j, h_{l'}]]^2, g_{j+1}, \dots, g_{j-1}\\
    &+\sum_{l=1}^k\sum_{i=1}^n \rho(\beta, [[h_l, f_i]]^1), f_{i+1}, \dots, f_{i-1}, [[h_l, f_i]]^2, h_{l+1}, \dots, h_{l-1}\\
    &+\sum_{l=1}^k\sum_{i=1}^n [[h_l, f_i]]^1, f_{i+1}, \dots, f_{i-1}, \rho(\beta, [[h_l, f_i]]^2), h_{l+1}, \dots, h_{l-1}\\
    &+\sum_{l=1}^k\sum_{i=1}^n\sum_{l'\neq l} [[h_l, f_i]]^1, f_{i+1}, \dots, f_{i-1}, [[h_l, f_i]]^2, h_{l+1}, \dots, \rho(\beta, h_{l'}), \dots, h_{l-1}\\
    &+\sum_{l=1}^k\sum_{i=1}^n\sum_{i'\neq i} [[h_l, f_i]]^1, f_{i+1}, \dots, \rho(\beta, f_{i'}), \dots, f_{i-1}, [[h_l, f_i]]^2, h_{l+1}, \dots, h_{l-1}\\
    &+\sum_{l=1}^k \sum_{i=1}^n \sum_{j=1}^m [[g_j, [[h_l, f_i]]^1]]^1, f_{i+1}, \dots, [[h_l, f_i]]^2, h_{l+1}, \dots, [[g_j, [[h_l, f_i]]^1]]^2, g_{j+1}, \dots, g_{j-1}\\
    &+\sum_{l=1}^k \sum_{i=1}^n \sum_{j=1}^m [[g_j, [[h_l, f_i]]^2]]^1, h_{l+1}, \dots, [[h_l, f_i]]^1, f_{i+1}, \dots, [[g_j, [[h_l, f_i]]^2]]^2, g_{j+1}, \dots, g_{j-1}\\
    &+\sum_{l=1}^k \sum_{i=1}^n \sum_{j=1}^m \sum_{l'\neq l} [[g_j, h_{l'}]]^1, h_{l'+1}, \dots, [[h_l, f_i]]^1, f_{i+1}, \dots, [[h_l, f_i]]^2, h_{l+1}, \dots, h_{l'-1},\\
    &\qquad\qquad\qquad\qquad\qquad [[g_j, h_{l'}]]^2, g_{j+1}, \dots, g_{j-1}\notag\\
    &+\sum_{l=1}^k \sum_{i=1}^n \sum_{j=1}^m \sum_{i'\neq i} [[g_j, f_{i'}]]^1, f_{i'+1}, \dots, [[h_l, f_i]]^2, h_{l+1}, \dots, [[h_l, f_i]]^1, f_{i+1}, \dots, f_{i'-1},\\
    &\qquad\qquad\qquad\qquad\qquad [[g_j, f_{i'}]]^2, g_{j+1}, \dots, g_{j-1}\notag\\
    \{c, \{a, b\}\} &= [\gamma, [\alpha, \beta]]\\
    &+\sum_{l=1}^k h_1, \dots, \overline{\rho}(h_l, [\alpha, \beta]), \dots, h_k\\
    &+ \sum_{j=1}^m g_1, \dots, \rho(\gamma, \rho(\alpha, g_j)), \dots, g_m\\
    &+\sum_{j=1}^m\sum_{j'\neq j} g_1, \dots, \rho(\gamma, g_{j'}), \dots, \rho(\alpha, g_j), \dots, g_m\\
    &+\sum_{j=1}^m \sum_{l=1}^k [[h_l, \rho(\alpha, g_j)]]^1, g_{j+1}, \dots, g_{j-1}, [[h_l, \rho(\alpha, g_j)]]^2, h_{l+1}, \dots, h_{l-1}\\
    &+\sum_{j=1}^m \sum_{l=1}^k \sum_{j' \neq j} [[h_l, g_{j'}]]^1, g_{j'+1}, \dots, \rho(\alpha, g_j), \dots ,g_{j'-1}, [[h_l, g_{j'}]]^2, h_{l+1}, \dots, h_{l-1}\\
    &+\sum_{i=1}^n f_1, \dots, \rho(\gamma, \overline{\rho}(f_i, \beta)), \dots, f_n\\
    &+ \sum_{i=1}^n \sum_{i' \neq i} f_1, \dots, \rho(\gamma, f_{i'}), \dots, \overline{\rho}(f_i, \beta), \dots, f_n\\
    &+ \sum_{i=1}^n \sum_{l=1}^k [[h_l, \overline{\rho}(f_i, \beta)]]^1, f_{i+1}, \dots, f_{i-1}, [[h_l, \overline{\rho}(f_i, \beta)]]^2, h_{l+1}, \dots, h_{l-1}\\
    &+\sum_{i=1}^n \sum_{l=1}^k \sum_{i'\neq i} [[h_l, f_{i'}]]^1, f_{i'+1}, \dots, \overline{\rho}(f_i, \beta), \dots, f_{i'-1}, [[h_l, f_{i'}]]^2, h_{l+1}, \dots, h_{l-1}\\
    &+\sum_{i=1}^n\sum_{j=1}^m \rho(\gamma, [[f_i, g_j]]^1), g_{j+1}, \dots, g_{j-1}, [[f_i, g_j]]^2, f_{i+1}, \dots, f_{i-1}\\
    &+\sum_{i=1}^n\sum_{j=1}^m [[f_i, g_j]]^1, g_{j+1}, \dots, g_{j-1}, \rho(\gamma, [[f_i, g_j]]^2), f_{i+1}, \dots, f_{i-1}\\
    &+\sum_{i=1}^n\sum_{j=1}^m\sum_{i'\neq i} [[f_i, g_j]]^1, g_{j+1}, \dots, g_{j-1}, [[f_i, g_j]]^2, f_{i+1}, \dots, \rho(\gamma, f_{i'}), \dots, f_{i-1}\\
    &+\sum_{i=1}^n\sum_{j=1}^m\sum_{j'\neq j} [[f_i, g_j]]^1, g_{j+1}, \dots, \rho(\gamma, g_{j'}), \dots, g_{j-1}, [[f_i, g_j]]^2, f_{i+1}, \dots, f_{i-1}\\
    &+\sum_{i=1}^n \sum_{j=1}^m \sum_{l=1}^k [[h_l, [[f_i, g_j]]^1]]^1, g_{j+1}, \dots, [[f_i, g_j]]^2, f_{i+1}, \dots, [[h_l, [[f_i, g_j]]^1]]^2, h_{l+1}, \dots, h_{l-1}\\
    &+\sum_{i=1}^n \sum_{j=1}^m \sum_{l=1}^k [[h_l, [[f_i, g_j]]^2]]^1, f_{i+1}, \dots, [[f_i, g_j]]^1, g_{j+1}, \dots, [[h_l, [[f_i, g_j]]^2]]^2, h_{l+1}, \dots, h_{l-1}\\
    &+\sum_{i=1}^n \sum_{j=1}^m \sum_{l=1}^k \sum_{i'\neq i} [[h_l, f_{i'}]]^1, f_{i'+1}, \dots, [[f_i, g_j]]^1, g_{j+1}, \dots, [[f_i, g_j]]^2, f_{i+1}, \dots, f_{i'-1},\\
    &\qquad\qquad\qquad\qquad\qquad [[h_l, f_{i'}]]^2, h_{l+1}, \dots, h_{l-1}\notag\\
    &+\sum_{i=1}^n \sum_{j=1}^m \sum_{l=1}^k \sum_{j'\neq j} [[h_l, g_{j'}]]^1, g_{j'+1}, \dots, [[f_i, g_j]]^2, f_{i+1}, \dots, [[f_i, g_j]]^1, g_{j+1}, \dots, g_{j'-1},\\
    &\qquad\qquad\qquad\qquad\qquad [[h_l, g_{j'}]]^2, h_{l+1}, \dots, h_{l-1}\notag
  \end{align}}
  This whole thing sums to zero by the following arguments: 
  \begin{itemize}
    \item Lines (1), (19), and (37) cancel out by Jacobi of $[-,-]$. 
    \item Lines (2), (21), and (43) cancel out by the Lie module condition.
    \item Lines (3), (25), and (38) cancel out by the Lie module condition.  
    \item Lines (4) and (26) cancel out by skew symmetry of $\rho$ and $\ol\rho$.
    \item Lines (5), (29), (30), and (45) cancel out by compatibility of $[[-,-]]$ with $\rho$. 
    \item Lines (6) and (31) cancel out by skew symmetry of $[[-,-]]$. 
    \item Lines (7), (20), and (39) cancel out by the Lie module condition. 
    \item Lines (8) and (40) cancel out by skew symmetry of $\rho$ and $\ol{\rho}$.
    \item Lines (9), (23), (47), and (48) cancel out by compatibility of $[[-,-]]$ with $\rho$. 
    \item Lines (10) and (50) cancel out by the skew symmetry of $\rho$ and $\ol\rho$. 
    \item Lines (11), (12), (27), and (41) cancel out by compatibility of $[[-,-]]$ with $\rho$. 
    \item Lines (13) and (42) cancel out by skew symmetry of $[[-,-]]$. 
    \item Lines (14) and (28) cancel out by skew symmetry of $\rho$ and $\ol\rho$. 
    \item Lines (15), (33), and (51) cancel out by the double Jacobi identity.
    \item Lines (16), (34), and (52) cancel out by the double Jacobi identity. 
    \item Lines (17) and (54) cancel out by skew symmetry of $[[-,-]]$. 
    \item Lines (18) and (35) cancel out by skew symmetry of $[[-,-]]$.
    \item Lines (22) and (44) cancel out by skew symmetry of $\rho$ and $\ol\rho$. 
    \item Lines (24) and (49) cancel out by skew symmetry of $[[-,-]]$. 
    \item Lines (32) and (46) cancel out by skew symmetry of $\rho$ and $\ol\rho$. 
    \item Lines (36) and (53) cancel out by skew symmetry of $[[-,-]]$. 
  \end{itemize}
  Thus, $\{-,-\}$ satisfies Jacobi and is a Lie bracket. 
\end{proof}

\begin{proof}[Proof of Theorem \ref{liecobracket}]
  Skew symmetry of $\Delta$ is obvious from the skew symmetry of $\delta, [[-,-]]$ and the relationship $\sigma = -\tau_{(12)} \sigma$. We write out the proof of the coJacobi identity in full detail below. We take $a = (\alpha; f_1, \dots, f_n)$. The notation $k < i < j$ implies the cyclic ordering of $k, i, j$, not just the linear order.
{\allowdisplaybreaks
\setcounter{equation}{0}
  \begin{align}
    (id + \tau_{(123)} + \tau_{(132)})&(\delta \otimes 1)(\delta)(a) = \delta(\delta(\alpha)^1)^1 \otimes \delta(\delta(\alpha)^1)^2 \otimes \delta(\alpha)^2\\
    +\sum_{i=1}^n &\delta(\sigma(f_i)^1)^1 \otimes \delta(\sigma(f_i)^1)^2 \otimes f_1, \dots, \sigma(f_i)^2, \dots, f_n\\
    &+ \sigma(\ol\sigma(f_i)^1)^1 \otimes f_1, \dots, \sigma(\ol\sigma(f_i)^1)^2, \dots, f_n \otimes \ol\sigma(f_i)^2\\
    &+ f_1, \dots, \ol\sigma(\ol\sigma(f_i)^1)^1, \dots, f_n \otimes \ol\sigma(\ol\sigma(f_i)^1)^1 \otimes \ol\sigma(f_i)^2\\
    &+ \sum_{j\neq i} \sigma(f_j)^1 \otimes f_1, \dots, \sigma(f_j)^2, \dots, \ol\sigma(f_i)^1, \dots, f_n \otimes \ol\sigma(f_i)^2\\
    & \qquad + f_1, \dots, \ol\sigma(f_j)^1, \dots, \ol\sigma(f_i)^1, \dots, f_n \otimes \ol\sigma(f_j)^2 \otimes \ol\sigma(f_i)^2\\
    &\qquad + [[\ol\sigma(f_i)^1, f_j]]^2, f_{i+1}, \dots, f_{j-1} \otimes [[\ol\sigma(f_i)^1, f_j]]^1, f_{j+1}, \dots, f_{i-1} \otimes \ol\sigma(f_i)^2\\
    &\qquad + [[f_j, \ol\sigma(f_i)^1]]^2, f_{j+1}, \dots, f_{i-1} \otimes [[f_j, \ol\sigma(f_i)^1]]^1, f_{i+1}, \dots, f_{j-1} \otimes \ol\sigma(f_i)^2\\
    & +\sum_{k<i<j} [[f_k, f_j]]^2, f_{k+1}, \dots, \ol\sigma(f_i)^1, \dots, f_{j-1} \otimes [[f_k, f_j]]^1, f_{j+1}, \dots, f_{k-1} \otimes \ol{\sigma}(f_i)^2\\
    &\qquad + [[f_j, f_k]]^2, f_{j+1}, \dots, f_{k-1} \otimes [[f_j, f_k]]^1, f_{k+1}, \dots, \ol\sigma(f_i)^1, \dots, f_{j-1} \otimes \ol\sigma(f_i)^2\\
    & +\sum_{k<j<i} [[f_k, f_j]]^2, f_{k+1}, \dots, f_{j-1} \otimes [[f_k, f_j]]^1, f_{j+1}, \dots, \ol\sigma(f_i)^1, \dots, f_{k-1} \otimes \ol{\sigma}(f_i)^2\\
    &\qquad + [[f_j, f_k]]^2, f_{j+1}, \dots, \ol\sigma(f_i)^1, \dots, f_{k-1} \otimes [[f_j, f_k]]^1, f_{k+1}, \dots, f_{j-1} \otimes \ol\sigma(f_i)^2\\
    + \sum_{i<j}^n &\sigma([[f_i, f_j]]^2)^1 \otimes \sigma([[f_i, f_j]]^2)^2, f_{i+1}, \dots, f_{j-1} \otimes [[f_i, f_j]]^1, f_{j+1}, \dots f_{i-1} \\
    &+\ol\sigma([[f_i, f_j]]^2)^1, f_{i+1}, \dots, f_{j-1} \otimes \ol\sigma([[f_i, f_j]]^2)^2 \otimes [[f_i, f_j]]^1, f_{j+1}, \dots f_{i-1}\\
    &+\sum_{i < k < j} \sigma(f_k)^1 \otimes [[f_i, f_j]]^2, f_{i+1}, \dots, \sigma(f_k)^2, \dots, f_{j-1} \otimes [[f_i, f_j]]^1, f_{j+1}, \dots, f_{i-1}\\
    &\qquad + [[f_i, f_j]]^2, f_{i+1}, \dots, \ol\sigma(f_k)^1, \dots, f_{j-1} \otimes \ol\sigma(f_k)^2 \otimes [[f_i, f_j]]^1, f_{j+1}, \dots, f_{i-1}\\
    &\qquad + \big[\big[[[f_i, f_j]]^2, f_k\big]\big]^2, f_{i+1}, \dots, f_{k-1} \otimes \big[\big[[[f_i, f_j]]^2, f_k\big]\big]^1, f_{k+1}, \dots, f_{j-1} \\&\qquad\qquad\qquad \otimes [[f_i, f_j]]^1, f_{j+1}, \dots, f_{i-1}\notag\\
    &\qquad + \big[\big[f_k, [[f_i, f_j]]^2\big]\big]^2, f_{k+1}, \dots, f_{j-1} \otimes \big[\big[f_k, [[f_i, f_j]]^2\big]\big]^1, f_{i+1}, \dots, f_{k-1} \\&\qquad\qquad\qquad \otimes [[f_i, f_j]]^1, f_{j+1}, \dots, f_{i-1}\notag\\
    &\qquad +\sum_{i<k<l<j}[[f_k, f_l]]^2, f_{k+1}, \dots, f_{l-1} \otimes [[f_k, f_l]]^1, f_{l+1}, \dots, [[f_i, f_j]]^2, \dots, f_{k-1} \\&\qquad\qquad\qquad \otimes [[f_i, f_j]]^1, f_{j+1}, \dots, f_{i-1}\notag\\
    &\qquad\qquad +[[f_l, f_k]]^2, f_{l+1}, \dots, [[f_i, f_j]]^2, \dots, f_{k-1} \otimes [[f_l, f_k]]^1, f_{k+1}, \dots, f_{l-1} \\&\qquad\qquad\qquad \otimes [[f_i, f_j]]^1, f_{j+1}, \dots, f_{i-1}\notag\\
    &+\sigma([[f_j, f_i]]^2)^1 \otimes \sigma([[f_j, f_i]]^2)^2, f_{j+1}, \dots, f_{i-1} \otimes [[f_j, f_i]]^1, f_{i+1}, \dots f_{j-1} \\
    &+\ol\sigma([[f_j, f_i]]^2)^1, f_{j+1}, \dots, f_{i-1} \otimes \ol\sigma([[f_j, f_i]]^2)^2 \otimes [[f_j, f_i]]^1, f_{i+1}, \dots f_{j-1}\\
    &+\sum_{j < k < i} \sigma(f_k)^1 \otimes [[f_j, f_i]]^2, f_{j+1}, \dots, \sigma(f_k)^2, \dots, f_{i-1} \otimes [[f_j, f_i]]^1, f_{i+1}, \dots, f_{j-1}\\
    &\qquad + [[f_j, f_i]]^2, f_{j+1}, \dots, \ol\sigma(f_k)^1, \dots, f_{i-1} \otimes \ol\sigma(f_k)^2 \otimes [[f_j, f_i]]^1, f_{i+1}, \dots, f_{j-1}\\
    &\qquad + \big[\big[[[f_j, f_i]]^2, f_k\big]\big]^2, f_{j+1}, \dots, f_{k-1} \otimes \big[\big[[[f_j, f_i]]^2, f_k\big]\big]^1, f_{k+1}, \dots, f_{i-1} \\&\qquad\qquad\qquad \otimes [[f_j, f_i]]^1, f_{i+1}, \dots, f_{j-1}\notag\\
    &\qquad + \big[\big[f_k, [[f_j, f_i]]^2\big]\big]^2, f_{k+1}, \dots, f_{i-1} \otimes \big[\big[f_k, [[f_j, f_i]]^2\big]\big]^1, f_{j+1}, \dots, f_{k-1} \\&\qquad\qquad\qquad \otimes [[f_j, f_i]]^1, f_{i+1}, \dots, f_{j-1}\notag\\
    &\qquad +\sum_{j<k<l<i}[[f_k, f_l]]^2, f_{k+1}, \dots, f_{l-1} \otimes [[f_k, f_l]]^1, f_{l+1}, \dots, [[f_j, f_i]]^2, \dots, f_{k-1} \\&\qquad\qquad\qquad \otimes [[f_j, f_i]]^1, f_{i+1}, \dots, f_{j-1}\notag\\
    &\qquad\qquad +[[f_l, f_k]]^2, f_{l+1}, \dots, [[f_j, f_i]]^2, \dots, f_{k-1} \otimes [[f_l, f_k]]^1, f_{k+1}, \dots, f_{l-1} \\&\qquad\qquad\qquad \otimes [[f_j, f_i]]^1, f_{i+1}, \dots, f_{j-1}\notag\\
    &+\delta(\alpha)^2 \otimes \delta(\delta(\alpha)^1)^1 \otimes \delta(\delta(\alpha)^1)^2\\
    +\sum_{i=1}^n & f_1, \dots, \sigma(f_i)^2, \dots, f_n \otimes \delta(\sigma(f_i)^1)^1 \otimes \delta(\sigma(f_i)^1)^2\\
    &+ \ol\sigma(f_i)^2 \otimes \sigma(\ol\sigma(f_i)^1)^1 \otimes f_1, \dots, \sigma(\ol\sigma(f_i)^1)^2, \dots, f_n\\
    &+ \ol\sigma(f_i)^2 \otimes f_1, \dots, \ol\sigma(\ol\sigma(f_i)^1)^1, \dots, f_n \otimes \ol\sigma(\ol\sigma(f_i)^1)^1\\
    &+ \ol\sigma(f_i)^2 \otimes \sigma(f_j)^1 \otimes f_1, \dots, \sigma(f_j)^2, \dots, \ol\sigma(f_i)^1, \dots, f_n\\
    & \qquad + \ol\sigma(f_i)^2 \otimes f_1, \dots, \ol\sigma(f_j)^1, \dots, \ol\sigma(f_i)^1, \dots, f_n \otimes \ol\sigma(f_j)^2\\
    &\qquad + \ol\sigma(f_i)^2 \otimes [[\ol\sigma(f_i)^1, f_j]]^2, f_{i+1}, \dots, f_{j-1} \otimes [[\ol\sigma(f_i)^1, f_j]]^1, f_{j+1}, \dots, f_{i-1}\\
    &\qquad + \ol\sigma(f_i)^2 \otimes [[f_j, \ol\sigma(f_i)^1]]^2, f_{j+1}, \dots, f_{i-1} \otimes [[f_j, \ol\sigma(f_i)^1]]^1, f_{i+1}, \dots, f_{j-1}\\
    & +\sum_{k<i<j} \ol{\sigma}(f_i)^2 \otimes [[f_k, f_j]]^2, f_{k+1}, \dots, \ol\sigma(f_i)^1, \dots, f_{j-1} \otimes [[f_k, f_j]]^1, f_{j+1}, \dots, f_{k-1}\\
    &\qquad + \ol\sigma(f_i)^2 \otimes [[f_j, f_k]]^2, f_{j+1}, \dots, f_{k-1} \otimes [[f_j, f_k]]^1, f_{k+1}, \dots, \ol\sigma(f_i)^1, \dots, f_{j-1}\\
    & +\sum_{k<j<i} \ol{\sigma}(f_i)^2 \otimes [[f_k, f_j]]^2, f_{k+1}, \dots, f_{j-1} \otimes [[f_k, f_j]]^1, f_{j+1}, \dots, \ol\sigma(f_i)^1, \dots, f_{k-1}\\
    &\qquad + \ol\sigma(f_i)^2 \otimes [[f_j, f_k]]^2, f_{j+1}, \dots, \ol\sigma(f_i)^1, \dots, f_{k-1} \otimes [[f_j, f_k]]^1, f_{k+1}, \dots, f_{j-1}\\
    + \sum_{i<j}^n & [[f_i, f_j]]^1, f_{j+1}, \dots f_{i-1} \otimes \sigma([[f_i, f_j]]^2)^1 \otimes \sigma([[f_i, f_j]]^2)^2, f_{i+1}, \dots, f_{j-1} \\
    &+ [[f_i, f_j]]^1, f_{j+1}, \dots f_{i-1} \otimes \ol\sigma([[f_i, f_j]]^2)^1, f_{i+1}, \dots, f_{j-1} \otimes \ol\sigma([[f_i, f_j]]^2)^2\\
    &+\sum_{i < k < j} [[f_i, f_j]]^1, f_{j+1}, \dots, f_{i-1} \otimes \sigma(f_k)^1 \otimes [[f_i, f_j]]^2, f_{i+1}, \dots, \sigma(f_k)^2, \dots, f_{j-1}\\
    &\qquad + [[f_i, f_j]]^1, f_{j+1}, \dots, f_{i-1} \otimes [[f_i, f_j]]^2, f_{i+1}, \dots, \ol\sigma(f_k)^1, \dots, f_{j-1} \otimes \ol\sigma(f_k)^2\\
    &\qquad + [[f_i, f_j]]^1, f_{j+1}, \dots, f_{i-1} \otimes \big[\big[[[f_i, f_j]]^2, f_k\big]\big]^2, f_{i+1}, \dots, f_{k-1} \\&\qquad\qquad\qquad\qquad \otimes \big[\big[[[f_i, f_j]]^2, f_k\big]\big]^1, f_{k+1}, \dots, f_{j-1} \notag\\
    &\qquad + [[f_i, f_j]]^1, f_{j+1}, \dots, f_{i-1} \otimes \big[\big[f_k, [[f_i, f_j]]^2\big]\big]^2, f_{k+1}, \dots, f_{j-1} \\&\qquad\qquad\qquad\qquad \otimes \big[\big[f_k, [[f_i, f_j]]^2\big]\big]^1, f_{i+1}, \dots, f_{k-1} \notag\\
    &\qquad +\sum_{i<k<l<j} [[f_i, f_j]]^1, f_{j+1}, \dots, f_{i-1} \otimes [[f_k, f_l]]^2, f_{k+1}, \dots, f_{l-1} \\&\qquad\qquad\qquad\qquad \otimes [[f_k, f_l]]^1, f_{l+1}, \dots, [[f_i, f_j]]^2, \dots, f_{k-1} \notag\\
    &\qquad\qquad + [[f_i, f_j]]^1, f_{j+1}, \dots, f_{i-1} \otimes [[f_l, f_k]]^2, f_{l+1}, \dots, [[f_i, f_j]]^2, \dots, f_{k-1} \\&\qquad\qquad\qquad\qquad \otimes [[f_l, f_k]]^1, f_{k+1}, \dots, f_{l-1} \notag\\
    &+\qquad [[f_j, f_i]]^1, f_{i+1}, \dots f_{j-1} \otimes \sigma([[f_j, f_i]]^2)^1 \otimes \sigma([[f_j, f_i]]^2)^2, f_{j+1}, \dots, f_{i-1} \\
    &+\qquad [[f_j, f_i]]^1, f_{i+1}, \dots f_{j-1} \otimes \ol\sigma([[f_j, f_i]]^2)^1, f_{j+1}, \dots, f_{i-1} \otimes \ol\sigma([[f_j, f_i]]^2)^2\\
    &+\sum_{j < k < i} [[f_j, f_i]]^1, f_{i+1}, \dots, f_{j-1} \otimes \sigma(f_k)^1 \otimes [[f_j, f_i]]^2, f_{j+1}, \dots, \sigma(f_k)^2, \dots, f_{i-1}\\
    &\qquad + [[f_j, f_i]]^1, f_{i+1}, \dots, f_{j-1} \otimes [[f_j, f_i]]^2, f_{j+1}, \dots, \ol\sigma(f_k)^1, \dots, f_{i-1} \otimes \ol\sigma(f_k)^2\\
    &\qquad + [[f_j, f_i]]^1, f_{i+1}, \dots, f_{j-1} \otimes \big[\big[[[f_j, f_i]]^2, f_k\big]\big]^2, f_{j+1}, \dots, f_{k-1} \\&\qquad\qquad\qquad\qquad\otimes \big[\big[[[f_j, f_i]]^2, f_k\big]\big]^1, f_{k+1}, \dots, f_{i-1} \notag\\
    &\qquad + [[f_j, f_i]]^1, f_{i+1}, \dots, f_{j-1} \otimes \big[\big[f_k, [[f_j, f_i]]^2\big]\big]^2, f_{k+1}, \dots, f_{i-1} \\&\qquad\qquad\qquad\qquad\otimes \big[\big[f_k, [[f_j, f_i]]^2\big]\big]^1, f_{j+1}, \dots, f_{k-1} \notag\\
    &\qquad +\sum_{j<k<l<i} [[f_j, f_i]]^1, f_{i+1}, \dots, f_{j-1} \otimes [[f_k, f_l]]^2, f_{k+1}, \dots, f_{l-1} \\&\qquad\qquad\qquad\qquad \otimes [[f_k, f_l]]^1, f_{l+1}, \dots, [[f_j, f_i]]^2, \dots, f_{k-1} \notag\\
    &\qquad\qquad + [[f_j, f_i]]^1, f_{i+1}, \dots, f_{j-1} \otimes [[f_l, f_k]]^2, f_{l+1}, \dots, [[f_j, f_i]]^2, \dots, f_{k-1} \\&\qquad\qquad\qquad\qquad \otimes [[f_l, f_k]]^1, f_{k+1}, \dots, f_{l-1} \notag\\
    &\delta(\delta(\alpha)^1)^2 \otimes \delta(\alpha)^2 \otimes \delta(\delta(\alpha)^1)^1\\
    +\sum_{i=1}^n & \delta(\sigma(f_i)^1)^2 \otimes f_1, \dots, \sigma(f_i)^2, \dots, f_n \otimes \delta(\sigma(f_i)^1)^1\\
    &+ f_1, \dots, \sigma(\ol\sigma(f_i)^1)^2, \dots, f_n \otimes \ol\sigma(f_i)^2 \otimes \sigma(\ol\sigma(f_i)^1)^1\\
    &+ \ol\sigma(\ol\sigma(f_i)^1)^1 \otimes \ol\sigma(f_i)^2 \otimes f_1, \dots, \ol\sigma(\ol\sigma(f_i)^1)^1, \dots, f_n\\
    &+ f_1, \dots, \sigma(f_j)^2, \dots, \ol\sigma(f_i)^1, \dots, f_n \otimes \ol\sigma(f_i)^2 \otimes \sigma(f_j)^1\\
    & \qquad + \ol\sigma(f_j)^2 \otimes \ol\sigma(f_i)^2 \otimes f_1, \dots, \ol\sigma(f_j)^1, \dots, \ol\sigma(f_i)^1, \dots, f_n\\
    &\qquad + [[\ol\sigma(f_i)^1, f_j]]^1, f_{j+1}, \dots, f_{i-1} \otimes \ol\sigma(f_i)^2 \otimes [[\ol\sigma(f_i)^1, f_j]]^2, f_{i+1}, \dots, f_{j-1}\\
    &\qquad + [[f_j, \ol\sigma(f_i)^1]]^1, f_{i+1}, \dots, f_{j-1} \otimes \ol\sigma(f_i)^2 \otimes [[f_j, \ol\sigma(f_i)^1]]^2, f_{j+1}, \dots, f_{i-1}\\
    & +\sum_{k<i<j} [[f_k, f_j]]^1, f_{j+1}, \dots, f_{k-1} \otimes \ol{\sigma}(f_i)^2 \otimes [[f_k, f_j]]^2, f_{k+1}, \dots, \ol\sigma(f_i)^1, \dots, f_{j-1}\\
    &\qquad + [[f_j, f_k]]^1, f_{k+1}, \dots, \ol\sigma(f_i)^1, \dots, f_{j-1} \otimes \ol\sigma(f_i)^2 \otimes [[f_j, f_k]]^2, f_{j+1}, \dots, f_{k-1}\\
    & +\sum_{k<j<i} [[f_k, f_j]]^1, f_{j+1}, \dots, \ol\sigma(f_i)^1, \dots, f_{k-1} \otimes \ol{\sigma}(f_i)^2 \otimes [[f_k, f_j]]^2, f_{k+1}, \dots, f_{j-1}\\
    &\qquad + [[f_j, f_k]]^1, f_{k+1}, \dots, f_{j-1} \otimes \ol\sigma(f_i)^2 \otimes [[f_j, f_k]]^2, f_{j+1}, \dots, \ol\sigma(f_i)^1, \dots, f_{k-1}\\
    + \sum_{i<j}^n & \sigma([[f_i, f_j]]^2)^2, f_{i+1}, \dots, f_{j-1} \otimes [[f_i, f_j]]^1, f_{j+1}, \dots f_{i-1} \otimes \sigma([[f_i, f_j]]^2)^1 \\
    &+ \ol\sigma([[f_i, f_j]]^2)^2 \otimes [[f_i, f_j]]^1, f_{j+1}, \dots f_{i-1} \otimes \ol\sigma([[f_i, f_j]]^2)^1, f_{i+1}, \dots, f_{j-1}\\
    &+\sum_{i < k < j} [[f_i, f_j]]^2, f_{i+1}, \dots, \sigma(f_k)^2, \dots, f_{j-1} \otimes [[f_i, f_j]]^1, f_{j+1}, \dots, f_{i-1} \otimes \sigma(f_k)^1\\
    &\qquad + \ol\sigma(f_k)^2 \otimes [[f_i, f_j]]^1, f_{j+1}, \dots, f_{i-1} \otimes [[f_i, f_j]]^2, f_{i+1}, \dots, \ol\sigma(f_k)^1, \dots, f_{j-1}\\
    &\qquad + \big[\big[[[f_i, f_j]]^2, f_k\big]\big]^1, f_{k+1}, \dots, f_{j-1} \otimes [[f_i, f_j]]^1, f_{j+1}, \dots, f_{i-1} \\&\qquad\qquad\qquad\qquad \otimes \big[\big[[[f_i, f_j]]^2, f_k\big]\big]^2, f_{i+1}, \dots, f_{k-1} \notag\\
    &\qquad + \big[\big[f_k, [[f_i, f_j]]^2\big]\big]^1, f_{i+1}, \dots, f_{k-1} \otimes [[f_i, f_j]]^1, f_{j+1}, \dots, f_{i-1} \\&\qquad\qquad\qquad\qquad\otimes \big[\big[f_k, [[f_i, f_j]]^2\big]\big]^2, f_{k+1}, \dots, f_{j-1} \notag\\
    &\qquad +\sum_{i<k<l<j} [[f_k, f_l]]^1, f_{l+1}, \dots, [[f_i, f_j]]^2, \dots, f_{k-1} \otimes [[f_i, f_j]]^1, f_{j+1}, \dots, f_{i-1} \\&\qquad\qquad\qquad\qquad\otimes [[f_k, f_l]]^2, f_{k+1}, \dots, f_{l-1} \notag\\
    &\qquad\qquad + [[f_l, f_k]]^1, f_{k+1}, \dots, f_{l-1} \otimes [[f_i, f_j]]^1, f_{j+1}, \dots, f_{i-1} \\&\qquad\qquad\qquad\qquad \otimes [[f_l, f_k]]^2, f_{l+1}, \dots, [[f_i, f_j]]^2, \dots, f_{k-1} \notag\\
    &+\qquad \sigma([[f_j, f_i]]^2)^2, f_{j+1}, \dots, f_{i-1} \otimes [[f_j, f_i]]^1, f_{i+1}, \dots f_{j-1} \otimes \sigma([[f_j, f_i]]^2)^1 \\
    &+\qquad \ol\sigma([[f_j, f_i]]^2)^2 \otimes [[f_j, f_i]]^1, f_{i+1}, \dots f_{j-1} \otimes \ol\sigma([[f_j, f_i]]^2)^1, f_{j+1}, \dots, f_{i-1}\\
    &+\sum_{j < k < i} [[f_j, f_i]]^2, f_{j+1}, \dots, \sigma(f_k)^2, \dots, f_{i-1} \otimes [[f_j, f_i]]^1, f_{i+1}, \dots, f_{j-1} \otimes \sigma(f_k)^1\\
    &\qquad + \ol\sigma(f_k)^2 \otimes [[f_j, f_i]]^1, f_{i+1}, \dots, f_{j-1} \otimes [[f_j, f_i]]^2, f_{j+1}, \dots, \ol\sigma(f_k)^1, \dots, f_{i-1}\\
    &\qquad + \big[\big[[[f_j, f_i]]^2, f_k\big]\big]^1, f_{k+1}, \dots, f_{i-1} \otimes [[f_j, f_i]]^1, f_{i+1}, \dots, f_{j-1} \\&\qquad\qquad\qquad\qquad \otimes \big[\big[[[f_j, f_i]]^2, f_k\big]\big]^2, f_{j+1}, \dots, f_{k-1} \notag\\
    &\qquad + \big[\big[f_k, [[f_j, f_i]]^2\big]\big]^1, f_{j+1}, \dots, f_{k-1} \otimes [[f_j, f_i]]^1, f_{i+1}, \dots, f_{j-1} \\&\qquad\qquad\qquad\qquad\otimes \big[\big[f_k, [[f_j, f_i]]^2\big]\big]^2, f_{k+1}, \dots, f_{i-1} \notag\\
    &\qquad +\sum_{j<k<l<i} [[f_k, f_l]]^1, f_{l+1}, \dots, [[f_j, f_i]]^2, \dots, f_{k-1} \otimes [[f_j, f_i]]^1, f_{i+1}, \dots, f_{j-1} \\&\qquad\qquad\qquad\qquad\otimes [[f_k, f_l]]^2, f_{k+1}, \dots, f_{l-1} \notag\\
    &\qquad\qquad + [[f_l, f_k]]^1, f_{k+1}, \dots, f_{l-1} \otimes [[f_j, f_i]]^1, f_{i+1}, \dots, f_{j-1} \\&\qquad\qquad\qquad\qquad\otimes [[f_l, f_k]]^2, f_{l+1}, \dots, [[f_j, f_i]]^2, \dots, f_{k-1} \notag
  \end{align}

  This sums to zero by the following arguments: 
  \begin{itemize}
    \item Lines (1), (29), and (57) cancel by coJacobi of the cobracket. 
    \item Lines (2), (31), and (60); (3), (32), and (59); and (4), (30), and (59) cancel out by the comodule condition. 
    \item Lines (5) and (34); (6) and (61); and (33) and (62) cancel out by the relation between $\sigma$ and $\ol\sigma$. 
    \item Lines (7) and (8) cancel with (42) and (50), and (69) and (77) by the compatibility of $[[-,-]]$ with $\sigma$ and swapping the roles of $i$ and $j$ in line (8). 
    \item Lines (35) and (36) cancel with (13) and (21), and (70) and (78) by the compatibility of $[[-,-]]$ with $\sigma$ and swapping the roles of $i$ and $j$ in line (36). 
    \item Lines (63) and (64) cancel with (41) and (49), and (14) and (22) by the compatibility of $[[-,-]]$ with $\sigma$ and swapping the roles of $i$ and $j$ in line (64) 
    \item Lines (9) and (71); (15) and (37); and (43) and (65) cancel out by the relation of $\sigma$ and $\ol\sigma$. 
    \item Lines (10) and (52); (38) and (80); and (24) and (66) cancel out by the skew symmetry of $[[-,-]]$. 
    \item Lines (11) and (44); (39) and (72); and (67) and (16) cancel out by the skew symmetry of $[[-,-]]$. 
    \item Lines (12) and (79); (23) and (40); and (51) and (68) cancel out by the relation of $\sigma$ and $\ol\sigma$. 
    \item Lines (17) and (25); (45) and (53); and (73) and (81) cancel out by coJacobi of $[[-,-]]$ and swapping the roles of $i$,$j$, and $k$ appropriately when needed. 
    \item Lines (18) and (26); (46) and (54); and (74) and (82) cancel out by the negative coJacobi of $[[-,-]]$ and swapping the roles of $i$,$j$, and $k$ appropriately when needed. 
    \item Lines (19) and (48); (47) and (76); and (20) and (75) cancel out by swapping the roles $i \leftrightarrow l$, and $j \leftrightarrow k$ in lines (48), (76), and (20) and then applying skew symmetry of $[[-,-]]$.
    \item Lines (27) and (56); (55) and (84); and (28) and (83) cancel out by swapping the roles $i \leftrightarrow k$ and $j \leftrightarrow l$ in lines (56), (84), and (28) and then applying skew symmetry of $[[-,-]]$.
  \end{itemize}
}
\end{proof}

\begin{proof}[Proof of Theorem \ref{liebialgebra}]
  Let $a = (\alpha; f_1, \dots, f_n)$ and $b = (\beta; g_1, \dots, g_n)$. We prove that $\Delta(\{a, b\}) = (ad_a \otimes 1 + 1 \otimes ad_a)\Delta(b) - (ad_b \otimes 1 + 1 \otimes ad_b)\Delta(a)$, where $ad_x(y) = [x, y]$. 

  First, we calculate the left hand side (LHS).
  
  {\allowdisplaybreaks
  \setcounter{equation}{0}
  \begin{align}
    \Delta(\{a, b\}) &= \delta([\alpha, \beta])\\
    + & \sum_{j=1}^m \bigg(\sigma_1(\rho(\alpha, g_j)) \otimes g_1 \dots, \sigma_2(\rho(\alpha, g_j)), \dots, g_m\\
    &\qquad + g_1, \dots, \ol{\sigma_1}(\rho(\alpha, g_j)), \dots, g_m \otimes \ol{\sigma_2}(\rho(\alpha, g_j))\\
    & \qquad + \sum_{l \neq j}^m \sigma_1(g_l) \otimes g_1, \dots, \sigma_2(g_l), \dots, \rho(\alpha, g_j), \dots g_m\\
    & \qquad\qquad + g_1, \dots, \ol{\sigma_1}(g_l), \dots, \rho(\alpha, g_j), \dots, g_m \otimes \ol{\sigma_2}{g_l}\\
    & \qquad + \sum_{l \neq j}^m [[\rho(\alpha, g_j), g_l]]_2, g_{j+1}, \dots, g_{l-1} \otimes [[\rho(\alpha, g_j)]]_1, g_{l+1}, \dots, g_{j-1}\\
    & \qquad\qquad + [[g_l, \rho(\alpha, g_j)]]_2, g_{l+1}, \dots, g_{j-1} \otimes [[g_l, \rho(\alpha, g_j)]]_1, g_{j+1}, \dots, g_{l-1}\\
    &\qquad + \sum_{l < j < k}^m [[g_l, g_k]]_2, g_{l+1}, \dots, \rho(\alpha, g_j), \dots, g_{k-1} \otimes [[g_l, g_k]]_1, g_{k+1}, \dots, g_{l-1}\\
    & \qquad \qquad + [[g_k, g_l]]_2, g_{k+1}, \dots, g_{l-1} \otimes [[g_k, g_l]]_1, g_{l+1}, \dots, \rho(\alpha, g_j), \dots, g_{k-1}\\
    & \qquad \sum_{l < k < j}^m [[g_l, g_k]]_2, g_{l+1}, \dots, g_{k-1} \otimes [[g_l, g_k]]_1, g_{k+1}, \dots, \rho(\alpha, g_j), \dots, g_{l-1}\\
    & \qquad \qquad + [[g_k, g_l]]_2, g_{k+1}, \dots, \rho(\alpha, g_j), \dots, g_{l-1} \otimes [[g_k, g_l]]_2, g_{l+1}, \dots, g_{k-1}\bigg)\\
    + & \sum_{i=1}^n\bigg( \sigma_1(\ol\rho(f_i, \beta) \otimes f_1, \dots, \sigma_2(\ol\rho(f_i, \beta)), \dots, f_n\\
    & \qquad + f_1, \dots, \ol{\sigma_1}(\ol\rho(f_i, \beta)) \dots, f_n \otimes \ol{\sigma_2}(\ol\rho(f_i, \beta))\\
    & \qquad + \sum_{k\neq j}^n \sigma_1(f_k) \otimes f_1, \dots, \sigma_2(f_k), \dots, \ol{\rho}(f_i, \beta), \dots, f_n\\
    & \qquad \qquad + f_1, \dots, \ol{\sigma_1}(f_k), \dots, \ol\rho(f_i, \beta), \dots, f_n \otimes \ol{\sigma_2}(f_k)\\
    &\qquad + \sum_{k \neq j}^n [[\ol\rho(f_i, \beta), f)k]]_2, f_{i+1}, \dots, f_{k-1} \otimes [[\ol\rho(f_i,\beta), f_k]]_1, f_{k+1}, \dots, f_{i-1}\\
    &\qquad \qquad + [[f_k, \ol\rho(f_i, \beta)]]_2, f_{k+1}, \dots, f_{i-1} \otimes [[f_k, \ol\rho(f_i, \beta)]]_1, f_{i+1}, \dots, f_{k-1} \\
    & \qquad \sum_{l < i < k}^n [[f_l, f_k]]_2, f_{l+1}, \dots, \ol\rho(f_i, \beta), \dots, f_{k-1} \otimes [[f_l, f_k]]_1, f_{k+1}, \dots, f_{i-1} \\
    & \qquad \qquad [[f_k, f_l]]_2, f_{k+1}, \dots, f_{l-1} \otimes [[f_l, f_k]]_1, f_{l+1}, \dots, \ol\rho(f_i, \beta), \dots, f_{k-1}\\
    & \qquad \sum_{l < k < i}^n [[f_l, f_k]]_2, f_{l+1}, \dots, f_{k-1} \otimes [[f_l, f_k]]_1, f_{k+1}, \dots, \ol\rho(f_i, \beta), \dots, f_{k-1}\\
    & \qquad \qquad [[f_k, f_l]]_2, f_{k+1}, \dots, \ol\rho(f_i, \beta), \dots, f_{l-1} \otimes [[f_k, f_l]]_1, f_{l+1}, \dots, f_{k-1}\bigg) \\
    + \sum_{i=1}^n & \sum_{j=1}^m \bigg(\sigma_1([[f_i, g_j]]_1) \otimes \sigma_2([[f_i, g_j]]_1), \dots, [[f_i, g_j]]_2, \dots, f_{i-1} \\
    & \qquad + \ol{\sigma_1}([[f_i, g_j]]_1), \dots, [[f_i, g_j]]_2, \dots, f_{i-1} \otimes \ol{\sigma_2}([[f_i, g_j]]_1)\\
    & \qquad + \sigma_1([[f_i, g_j]]_2) \otimes [[f_i, g_j]]_1, \dots, \sigma_2([[f_i, g_j]]_2), \dots, f_{i-1} \\
    & \qquad + [[f_i, g_j]]_1, \dots, \ol{\sigma_1}([[f_i, g_j]]_2), \dots, f_{i-1} \otimes \ol{\sigma_2}([[f_i, g_j]]_2) \\
    & \qquad + \sum_{k\neq i}^n \sigma_1(f_k) \otimes [[f_i, g_j]], \dots, [[f_i, g_j]]_2, \dots, \sigma_2(f_k), \dots, f_{i-1} \\
    & \qquad \qquad + [[f_i, g_j]]_1, \dots, [[f_i, g_j]]_2,\dots, \ol{\sigma_1}(f_k), \dots, f_{i-1} \otimes \ol{\sigma_2}(f_k) \\
    & \qquad + \sum_{l \neq j}^m \sigma_1(g_l) \otimes [[f_i, g_j]]_1, \dots, \sigma_2(g_l), \dots, [[f_i, g_j]]_2, \dots, f_{i-1} \\
    & \qquad \qquad + [[f_i, g_j]]_1, \dots, \ol{\sigma_1}(g_l), \dots, [[f_i, g_j]]_2, \dots, f_{i-1} \otimes \ol{\sigma_2}(g_l) \\
    & \qquad + \bigdb{[[[f_i, g_j]]_1, [[f_i, g_j]]_2}_2, g_{j+1}, \dots, g_{j-1} \otimes \bigdb{[[f_i, g_j]]_1, [[f_i, g_j]]_2}_1, f_{i+1}, \dots, f_{i-1}\\
    & \qquad \bigdb{[[f_i, g_j]]_2, [[f_i, g_j]]_1}_2, f_{i+1}, \dots, f_{i-1} \otimes \big[\big[[[f_i, g_j]]_2, [[f_i, g_j]]_1\big]\big]_1, g_{i+1}, \dots, g_{j-1} \\
    & \qquad + \sum_{l \neq j}^m \bigdb{[[f_i, g_j]]_1, g_l}_2, g_{j+1}, \dots, g_{l-1} \otimes \bigdb{[[f_i, g_j]]_1, g_l}_1, g_{l+1}, \dots, [[f_i, g_j]]_2, \dots, f_{i-1}\\
    & \qquad \qquad + \bigdb{g_l, [[f_i, g_j]]_1}_2, g_{l+1}, \dots, [[f_i, g_j]]_2, \dots, f_{i-1} \otimes \bigdb{g_l, [[f_i, g_j]]_1}_1, g_{i+1}, \dots, g_{l-1}\\
    &\qquad \qquad + \bigdb{[[f_i, g_j]]_2, g_l}_2, f_{i+1}, \dots, [[f_i, g_j]]_1, \dots, g_{l-1} \otimes \bigdb{[[f_i, g_j]]_2, g_l}_1, g_{l+1}, \dots, g_{j-1}\\
    & \qquad \qquad + \bigdb{g_l, [[f_i, g_j]]_2}_2, g_{l+1}, \dots, g_{j-1} \otimes \bigdb{g_l, [[f_i, g_j]]_2}_1, f_{i+1}, \dots, [[f_i, g_j]]_1, \dots, g_{l-1} \\
    &\qquad +\sum_{k \neq i}^n \bigdb{[[f_i, g_j]]_1, f_k}_2, g_{j+1}, \dots, [[f_i, g_j]]_2, \dots, f_{k-1} \otimes \bigdb{[[f_i, g_j]]_1, f_k}_1, f_{k+1}, \dots, f_{i-1} \\
    &\qquad \qquad + \bigdb{f_k, [[f_i, g_j]]_1}_2, f_{k+1}, \dots, f_{i-1} \otimes \bigdb{f_k, [[f_i, g_j]]_1}_1, g_{j-1}, \dots, [[f_i, g_j]]_2, \dots, f_{k-1} \\
    & \qquad \qquad + \bigdb{[[f_i, g_j]]_2, f_k}_2, f_{i+1}, \dots, f_{k-1} \otimes \bigdb{[[f_i, g_j]]_2, f_k}_1, f_{k+1}, \dots, [[f_i, g_j]]_1, \dots, g_{j-1} \\
    & \qquad \qquad + \bigdb{f_k, [[f_i, g_j]]_2}_2, f_{k+1}, \dots, [[f_i, g_j]]_1, \dots, g_{j-1} \otimes \bigdb{f_k, [[f_i, g_j]]_2}_2, f_{i+1}, \dots, f_{k-1} \\
    & \qquad + \sum_{l < k \mathrm{~rel~}i} [[f_l, f_k]]_2, f_{l+1}, \dots, f_{k-1} \otimes [[f_l, f_k]]_1, f_{k+1}, \dots, [[f_i, g_j]]_1, \dots, [[f_i, g_j]]_2, \dots, f_{l-1}\\
    & \qquad \qquad + [[f_k, f_l]]_2, f_{k+1}, \dots, [[f_i, g_j]]_1, \dots, [[f_i, g_j]]_2, \dots, f_{l-1} \otimes [[f_k, f_l]]_1, f_{l+1}, \dots, f_{k-1}\\
    & \qquad + \sum_{l < k \mathrm{~rel~}j} [[g_l, g_k]]_2, g_{l+1}, \dots, g_{k-1} \otimes [[g_l, g_k]]_1, g_{k+1}, \dots, [[f_i, g_j]]_2, \dots, [[f_i, g_j]]_1, \dots, g_{l-1}\\
    &\qquad \qquad + [[g_k, g_l]]_2, g_{k+1}, \dots, [[f_i, g_j]]_2, \dots, [[f_i, g_j]]_1, \dots, g_{l-1} \otimes [[g_k, g_l]]_1, g_{l+1}, \dots, g_{k-1}\\
    & \quad + \sum_{l\neq j}^m \sum_{k\neq i}^n [[g_l, f_k]]_2, g_{l+1}, \dots, [[f_i, g_j]]_2, \dots, f_{k-1} \otimes [[g_l, f_k]]_1, f_{k+1}, \dots, [[f_i, g_j]]_1, \dots, g_{l-1}  \\
    & \qquad \qquad + [[f_k, g_l]]_2, f_{k+1}, \dots, [[f_i, g_j]]_1, \dots, g_{l-1} \otimes [[f_k, g_l]]_1, g_{l+1}, \dots, [[f_i, g_j]]_2, \dots, f_{k-1}\bigg)
  \end{align}

  Now we calculate the right hand side (RHS): 

  \begin{align}
    (ad_a\otimes 1)\Delta(b) = &\qquad [\alpha, \delta_1(\beta)] \otimes \delta_2(\beta)\\
    & + \sum_{i=1}^n f_1, \dots, \ol{\rho}(f_i, \delta_1(\beta)), \dots, f_n \otimes \delta_2(\beta)\\
    & + \sum_{l=1}^m \bigg( [\alpha, \sigma_1(g_l)] \otimes g_1, \dots, \sigma_2(g_l), \dots, g_m \\
    & \qquad + \sum_{i=1}^n f_1, \dots, \ol{\rho(f_i, \sigma_1(g_l))}, \dots, f_n \otimes g_!, \dots, \sigma_2(g_l), \dots, g_m  \\
    & \qquad + g_1, \dots, \rho(\alpha, \ol{\sigma_1}(g_l)), \dots, g_m \otimes \ol{\sigma_2}(g_l) \\
    & \qquad + \sum_{j\neq l}^m g_1, \dots, \ol{\sigma_1}(g_l), \dots, \rho(\alpha, g_j), \dots, g_m \otimes \ol{\sigma_2}(g_l)\bigg) \\
    & \qquad + \sum_{i=1}^n [[f_i, \ol{\sigma_1}(g_l)]]_1, \dots, g_{l-1} [[f_i, \ol\sigma_{1}(g_l)]]_2, \dots, f_{i-1} \otimes \ol\sigma_2(g_l) \\
    & \qquad + \sum_{i=1}^n \sum_{j \neq l}^m [[f_i, g_j]]_1, g_{j+1}, \dots, \ol\sigma_1(g_l), \dots, g_{j-1}, [[f_i, g_j]]_2, \dots, f_{i-1} \otimes \ol\sigma_2(g_l) \\
    & + \sum_{l < k}^m \bigg(\rho(\alpha, [[g_l, g_k]]_2), g_{l+1}, \dots, g_{k-1} \otimes [[g_l, g_k]]_1, \dots, g_{l-1}\\
    & \qquad + \sum_{l < j < k} [[g_l, g_k]]_2, \dots, \rho(\alpha, g_j), \dots, g_{k-1} \otimes [[g_l, g_k]]_1, \dots, g_{l-1} \\
    & \qquad + \sum_{i=1}^n \bigdb{f_i, [[g_l, g_k]]_2}_1, \dots, g_{k-1}, \bigdb{f_i, [[g_l, g_k]]_2}_2, \dots, f_{i-1} \otimes [[g_l, g_k]]_1, \dots, g_{l-1}\\
    & \quad + \sum_{i=1}^n \sum_{l < j < k} [[f_i, g_j]]_1, \dots, [[g_l, g_k]]_2, \dots, g_{j-1}, [[f_i, g_j]]_2, \dots, f_{i-1} \otimes [[g_l, g_k]]_1, \dots, g_{l-1}\\
    & \qquad \rho(\alpha, [[g_k, g_l]]_2), g_{k+1}, \dots, g_{l-1} \otimes [[g_k, g_l]]_1, \dots, g_{k-1} \\
    & \qquad + \sum_{j < l < k} [[g_k, g_l]]_2, \dots, \rho(\alpha, g_j), \dots, g_{l-1} \otimes [[g_k, g_l]]_1, \dots, g_{k-1} \\
    & \qquad + \sum_{i=1}^n \bigdb{f_i, [[g_k, g_l]]_2}_1, \dots, g_{l-1}, \bigdb{f_i, [[g_k, g_l]]_2}_2, \dots, f_{i-1} \otimes [[g_k, g_l]]_1, \dots, g_{k-1}\\
    & \quad + \sum_{i=1}^n \sum_{j < k < l} [[f_i, g_j]]_1, \dots, [[g_k, g_l]]_2, \dots, g_{j-1}, [[f_i, g_j]]_2, \dots, f_{i-1} \otimes [[g_k, g_l]]_1, \dots, g_{k-1}
  \end{align}

  \begin{align}
    (1 \otimes ad_a)\Delta(b) = & \qquad \delta_1(\beta) \otimes [\alpha, \delta_2(\beta)]  \\
    & + \sum_{i=1}^n \delta_1(\beta) \otimes f_1, \dots, \ol\rho(f_1, \delta_2(\beta)), \dots, f_n \\
    & +\sum_{l=1}^m \bigg(\sigma_1(g_l) \otimes g_1, \dots, \rho(\alpha, \sigma_2(g_l)), \dots, g_m \\
    & \qquad + \sum_{j \neq l}^m \sigma_1(g_l) \otimes g_1, \dots, \sigma_2(g_l), \dots, \rho(\alpha, g_j), \dots, g_m\\
    & \qquad \sum_{i=1}^n \sigma_1(g_l) \otimes [[f_i, \sigma_2(g_l)]]_1, g_{l+1}, \dots, g_{l-1}, [[f_i, \sigma_2(g_l)]]_2, \dots, f_{i-1} \\
    & \qquad \sum_{i=1}^n \sum_{j \neq l}^m \sigma_1(g_l) \otimes [[f_i, g_j]]_1, \dots, \sigma_2(g_l), \dots, g_{j-1}, [[f_i, g_j]]_2, \dots, f_{i-1}\\
    & \qquad + g_1, \dots, \ol{\sigma_1}(g_l), \dots, g_{m} \otimes [\alpha, \sigma_2(g_l)]\\
    &\qquad \sum_{i=1}^n g_1, \dots, \ol{\sigma_1}, \dots, g_m \otimes f_1, \dots, \ol\rho(f_i, \ol{\sigma_2}(g_l)), \dots, f_{n}\bigg) \\
    &+ \sum_{l < k} \bigg( [[g_l, g_k]]_2, g_{l+1}, \dots, g_{k-1} \otimes \rho(\alpha, [[g_l, g_k]]_1), g_{k+1}, \dots, g_{l-1} \\
    & \qquad + \sum_{j < l < k}^m [[g_l, g_k]]_2, \dots, g_{k-1} \otimes [[g_l, g_k]]_1, \dots, \rho(\alpha, g_j), \dots, g_{l-1} \\
    & \qquad + \sum_{i=1}^n [[g_l, g_k]]_2, \dots, g_{k-1} \otimes \bigdb{f_i, [[g_l, g_k]]_1}_1, \dots, g_{l-1}, \bigdb{f_i, [[g_l, g_k]]_1}_2, \dots, f_i \\
    & \quad + \sum_{i=1}^n\sum_{j<l<k}^m [[g_l, g_k]]_2, \dots, g_{k-1} \otimes [[f_i, g_j]]_1, \dots, [[g_l, g_k]]_1, \dots, g_{j-1}, [[f_i, g_j]]_2, \dots, f_{i-1} \\
    &\qquad+ [[g_k, g_l]]_2, g_{k+1}, \dots, g_{l-1} \otimes \rho(\alpha, [[g_k, g_l]]_1), g_{l+1}, \dots, g_{k-1} \\
    &\qquad + \sum_{l<j<k}^m [[g_k, g_l]]_2, \dots, g_{l-1} \otimes [[g_k, g_l]]_1, \dots, \rho(\alpha, g_j), \dots, g_{k-1} \\
    &\qquad + \sum_{i=1}^n [[g_k, g_l]]_2, \dots, g_{l-1} \otimes \bigdb{f_i, [[g_l, g_k]]_1}_1, \dots, g_{k-1}, \bigdb{f_i, [[g_l, g_k]]_1}_2, \dots, f_{i-1}\\
    &\quad + \sum_{i=1}^n \sum_{l<j<k}^m [[g_k, g_l]]_2, \dots, g_{l-1} \otimes [[f_i, g_j]]_1, \dots, [[g_l, g_k]]_1, \dots, g_{k-1}, [[f_i, g_j]]_2, \dots, f_{i-1}
  \end{align}

  \begin{align}
    -(ad_b \otimes 1)\Delta(a) = & \qquad [\beta, \delta_1(\alpha)] \otimes \delta_2(\alpha) \\
    & + \sum_{j=1}^n g_j, \dots, \ol\rho(g_j, \delta_1(\alpha)), \dots, g_m \otimes \delta_2(\alpha) \\
    & + \sum_{k=1}^n \bigg([\beta, \sigma_1(f_k)] \otimes f_1, \dots, \sigma_2(f_k), \dots, f_n \\
    & \qquad + \sum_{j=1}^m g_!, \dots, \ol\rho(g_j, \sigma_1(f_k)), \dots, g_m \otimes f_1, \dots, \sigma_2(f_k), \dots, f_n \\
    & \qquad + f_1, \dots, \rho(\beta, \ol{\sigma_1}(f_k)), \dots, f_n \otimes \ol{\sigma_2}(f_k) \\
    & \qquad + \sum_{i \neq k}^n f_1, \dots, \ol{\sigma_1}(f_k), \dots, \rho(\beta, f_i), \dots, f_n \otimes \ol{\sigma_2}(f_k) \\
    & \qquad + \sum_{j=1}^m [[g_j, \ol{\sigma_2}(f_k)]]_1, \dots, f_{k-1}, [[g_j, \ol{\sigma_1}(f_k)]]_2, \dots, g_{j-1} \otimes \ol{\sigma_2}(f_k) \\
    & \quad + \sum_{j=1}^m\sum_{i \neq k}^n [[g_j, f_i]]_1, f_{i+1}, \dots, \ol{\sigma_1}(f_k), \dots, f_{i-1}, [[g_j, f_i]]_2, \dots, g_{j-1} \otimes \ol{\sigma_2}(f_k)\bigg) \\
    & + \sum_{l<k}^n \bigg( \rho(\beta, [[f_l, f_k]]_2), f_{l+1}, \dots, f_{k-1} \otimes [[f_l, f_k]]_1, f_{k+1}, \dots, f_{l-1} \\
    & \qquad + \sum_{l < i < k}^n [[f_l, f_k]]_2, \dots, \rho(\beta, f_i), \dots, f_{k-1} \otimes [[f_l, f_k]]_1, \dots, f_{l-1} \\
    & \qquad \sum_{j=1}^m \bigdb{g_j, [[f_l, f_k]]_2}_1, \dots, f_{k-1}, \bigdb{g_j, [[f_l, f_k]]_2}_2, \dots, g_{j-1} \otimes [[f_l, f_k]]_1, \dots, f_{l-1} \\
    & \quad \sum_{j = 1}^n \sum_{l < i < k} [[g_j, f_i]]_1, \dots, [[f_l, f_k]]_2, \dots, [[g_j, f_i]]_2, \dots, g_{j-1} \otimes [[f_l, f_k]]_1, \dots, f_{l-1} \\
    & \qquad \rho(\beta, [[f_k, f_l]]_2), f_{k+1}, \dots, f_{l-1} \otimes [[f_k, f_l]]_1, f_{l+1}, \dots, f_{k-1} \\
    & \qquad \sum_{i<l<k} [[f_k, f_l]]_2, \dots, \rho(\beta, f_i), \dots, f_{l-1} \otimes [[f_k, f_l]]_1, \dots, f_{k-1}  \\
    & \qquad \sum_{j=1}^m \bigdb{g_j, [[f_k, f_l]]_2}_1, \dots, f_{l-1}, \bigdb{g_j, [[f_k, f_l]]_2}_2, \dots, g_{j-1} \otimes [[f_k, f_l]]_1, \dots, f_{k-1}\\
    & \quad \sum_{j=1}^m \sum_{i<l<k} [[g_j, f_i]]_1, \dots, [[f_k, f_l]]_2, \dots, [[g_j, f_i]]_2, \dots, g_{j-1} \otimes [[f_k, f_l]]_1, \dots, f_{l-1}\bigg) 
  \end{align}

  \begin{align}
    -(1\otimes ad_b)\Delta(a) =& \qquad \delta_1(\alpha) \otimes [\beta, \delta_2(\alpha)] \\
    & + \sum_{j=1}^m \delta_1(\alpha) \otimes g_1,\dots, \ol{\rho}(g_j, \delta_2(\alpha)), \dots, g_m \\
    & + \sum_{k=1}^n\bigg(\sigma_1(f_k) \otimes f_, \dots, \rho(\beta, \sigma_2(f_k)), \dots, f_n \\
    & \qquad + \sum_{i\neq k}^n \sigma_1(f_k), \dots, f_1, \dots, \sigma_2(f_k), \dots, \rho(\beta, f_i), \dots, f_n\\
    & \qquad + \sum_{j=1}^m \sigma_1(f_k) \otimes [[g_j, \sigma_2(f_k)]]_1, f_{k+1}, \dots, f_{k-1}, [[g_j, \sigma_2(f_k)]]_2, \dots, g_{j-1} \\
    & \quad + \sum_{j-1}^n \sum_{i\neq k}^n \sigma_1(f_k) \otimes [[g_j, f_i]]_1, \dots, \sigma_2(f_k), \dots, [[g_j, f_i]]_2, \dots, g_{j-1} \\
    & \qquad + f_1, \dots, \ol{\sigma_1}(f_k), \dots, f_n \otimes [\beta, \ol{\sigma_2}(f_k)]\\
    & \qquad + \sum_{j=1}^m f_1, \dots, \ol{\sigma_1}(f_k), \dots, f_n \otimes g_1, \dots, \ol\rho(g_, \ol{\sigma_2}(f_k)), \dots, g_m\bigg) \\
    & + \sum_{l<k}^n \bigg([[f_l, f_k]]_2, f_{l+1}, \dots, f_{k-1} \otimes \rho(\beta, [[f_l, f_k]]_1), f_{k+1}, \dots, f_{l-1} \\
    & \qquad + \sum_{i<l<k} [[f_l, f_k]]_2, \dots, f_{k-1} \otimes [[f_l, f_k]]_1, \dots, \rho(\beta, f_i), \dots, f_{l-1} \\
    &\qquad + \sum_{j=1}^m [[f_l, f_k]]_2, \dots, f_{k-1} \otimes \bigdb{g_j, [[f_l, f_k]]_1}_1,\dots, f_{l-1}, \bigdb{g_j, [[f_l, f_k]]_1}_2, \dots, g_{j-1} \\
    & \quad + \sum_{j=1}^n \sum_{i<l<k} [[f_l, f_k]]_2, \dots, f_{k-1} \otimes [[g_j, f_i]]_1, \dots, [[f_l, f_k]]_1, \dots, [[g_j, f_i]]_2, \dots, g_{j-1} \\
    & \qquad + [[f_k, f_l]]_2, f_{k+1}, \dots, f_{l-1} \otimes \rho(\beta, [[f_k, f_l]]_1) f_{l+1}, \dots, f_{k-1} \\
    &\qquad + \sum_{l<i<k} [[f_k, f_l]]_2, \dots, f_{l-1} \otimes [[f_k, f_l]]_1, \dots, \rho(\beta, f_i), \dots, f_{l-1} \\
    & \qquad + \sum_{j=1}^m [[f_k, f_l]]_2, \dots, f_{l-1} \otimes \bigdb{g_j, [[f_k, f_l]]_1}_1, \dots, f_{k-1}, \bigdb{g_j, [[f_k ,f_l]]_1}_2, \dots, g_{j-1}\\
    & \quad + \sum_{j=1}^m \sum_{l<i<k} [[f_k, f_l]]_2, \dots, f_{l-1}\otimes [[g_j, f_i]]_1, \dots, [[f_k, f_l]]_1, \dots, [[g_j, f_i]]_2, \dots, g_{j-1}
  \end{align}
  }

  The RHS and LHS are equal for the following reasons: 
  \begin{itemize}
    \item Line (1) is equal to the sum of lines (46), (62), (78), and (94) by the Lie bialgebra condition on $[-,-]$ and $\delta$. 
    \item Line (2) is equal to the sum of lines (48), (64), and (95) by the bimodule condition on $\sigma$ and $\rho$.
    \item Line (3) is equal to the sum of lines (50), (68), and (79) by the bomiodule condition on $\sigma$ and $\rho$. 
    \item Line (4) is equal to line (65) by changing the order of summation. 
    \item Line (5) is equal to line (51) by changing the order of summation.
    \item The sum of lines (6) and (7) is equal to the sum of lines (54),  (58), (70), and (74) by the bimodule condition on $\rho$ and $[[-,-]]$. 
    \item Line (8) is equal to line (55) by changing the order of summation. 
    \item Line (9) is equal to line (75) by changing the order of summation. 
    \item Line (10) is equal to line (71) by changing the order of summation. 
    \item Line (11) is equal to line (59) by changing the order of summation.
    \item Line (12) is equal to the sum of lines (63), (80), and (96) by the bimodule condition on $\sigma$ and $\rho$. 
    \item Line (13) is equal to the sum of lines (47), (82), and (100) by the bimodule condition on $\sigma$ and $\rho$. 
    \item Line (14) is equal to line (97) by changing the order of summation.
    \item Line (15) is equal to line (83) by changing the order of summation. 
    \item The sum of lines (16) and (17) is equal to the sum of lines (86), (90), (102), and (106) by the bimodule condition on $\rho$ and $[[-,-]]$. 
    \item Line (18) is equal to line (87) by changing the order of summation. 
    \item Line (19) is equal to line (107) by changing the order of summation. 
    \item Line (20) is equal to line (103) by changing the order of summation.
    \item Line (21) is equal to line (91) by changing the order of summation. 
    \item The sum of lines (22) and (25) is equal to the sum of lines (66) and (98) by the bimodule condition on $\sigma$ and $[[-,-]]$. 
    \item The sum of lines (23) and (24) is equal to the sum of lines (84) and (52) by the bimodule condition on $\sigma$ and $[[-,-]]$. 
    \item Line (26) is equal to line (99) by changing the order of summation. 
    \item Line (27) is equal to line (85) by changing the order of summation. 
    \item Line (28) is equal to line (67) by changing the order of summation.
    \item Line (29) is equal to line (53) by changing the order of summation. 
    \item Line (30) is equal to the sum of lines (69) and (81) by the bimodule condition on $[[-,-]]$. 
    \item Line (31) is equal to the sum of lines (49) and (101) by the bimodule condition on $[[-,-]]$. 
    \item The sum of lines (32) and (35) is equal to the sum of lines (72) and (76) by the double Jacobi identity. 
    \item The sum of lines (33) and (34) is equal to the sum of lines (56) and (60) by the double Jacobi identity. 
    \item The sum of lines (36) and (39) is equal to the sum of lines (88) and (92) by the double Jacobi identity.
    \item The sum of lines (37) and (38) is equal to the sum of lines (104) and (108) by the double Jacobi identity.
    \item Line (40) is equal to the sum of lines (105) and (109) by changing the order of summation. 
    \item Line (41) is equal to the sum of lines (89) and (93) by changing the order of summation. 
    \item Line (42) is equal to the sum of lines (73) and (77) by changing the order of summation. 
    \item Line (43) is equal to the sum of lines (57) and (61) by changing the order of summation. 
    \item Lines (44) and (45) sum to zero. This can be observed by switching the role of $i, j$ with the role of $l, k$ and observing that the four terms produced cancel each other out. 
\end{itemize}
  
  Thus, we have proven that $(W, \{-,-\}, \Delta)$ is a Lie bialgebra. 

  Now, if we assume that $(L, [-,-], \delta, M, \rho, \sigma, [[-,-]])$ is an \textit{involutive }double Lie bimodule, we will prove that $(W, \{-,-\}, \Delta)$ is involutive: 

  {\allowdisplaybreaks
  \setcounter{equation}{0}
  \begin{align}
    \{\Delta(a)\} &= \{\delta(a)^1, \delta(a)^2\}\\
    &+ \sum_{i=1}^n \bigg(f_1, \dots, \rho(\sigma(f_i)^1, \sigma(f_i)^2), \dots, f_n\\
    &\qquad + \sum_{i \neq j}^n f_i, \dots, \rho(\sigma(f_i)^1, f_j), \dots, \sigma(f_i)^2, \dots, f_n\\
    &\qquad + f_1, \dots, \ol\rho(\ol\sigma(f_i)^1, \ol\sigma(f_i)^2), \dots, f_n\\
    &\qquad \sum_{i\neq j}^n f_1, \dots, \ol\rho(f_j, \ol\sigma(f_i)^2), \dots, \ol\sigma(f_i)^1, \dots, f_n\bigg)\\
    &+ \sum_{i < j}^n \bigg(\big[\big[[[f_i, f_j]]^2, [[f_i, f_j]]^1\big]\big]^1, f_{i+1}, \dots, f_{j-1}, \big[\big[[[f_i, f_j]]^2, [[f_i, f_j]]^1\big]\big]^2, f_{j+1}, \dots, f_{i-1}\\
    &\qquad + \sum_{j < k < i} \big[\big[[[f_i, f_j]]^2, f_k\big]\big]^1. f_{k+1}, \dots, [[f_i, f_j]]^1, \dots, f_{k-1}, \big[\big[[[f_i, f_j]]^2, f_k\big]\big]^2, f_{i+1}, \dots, f_{j-1}\\
    &\qquad + \sum_{i < l < j} \big[\big[f_l, [[f_i, f_j]]^1]]^1, f_{j+1}, \dots, f_{i-1}, \big[\big[f_l, [[f_i, f_j]]^1]]^2, f_{l+1}, \dots, [[f_i, f_j]]^2, \dots, f_{l-1}\\
    &\qquad + \sum_{j < k < i}\sum_{i < l < j} [[f_l, f_k]]^1, f_{k+1}, \dots, [[f_i, f_j]]^1, \dots, f_{k-1}, [[f_l, f_k]]^2, f_{l+1}, \dots, [[f_i, f_j]]^2, \dots, f_{l-1}\\
    &\qquad +\big[\big[[[f_j, f_i]]^2, [[f_j, f_i]]^1\big]\big]^1, f_{j+1}, \dots, f_{i-1}, \big[\big[[[f_j, f_i]]^2, [[f_j, f_i]]^1\big]\big]^2, f_{i+1}, \dots, f_{j-1}\\
    &\qquad + \sum_{i < l < j} \big[\big[[[f_j, f_i]]^2, f_l\big]\big]^1. f_{l+1}, \dots, [[f_j, f_i]]^1, \dots, f_{l-1}, \big[\big[[[f_j, f_i]]^2, f_l\big]\big]^2, f_{j+1}, \dots, f_{i-1}\\
    &\qquad + \sum_{j < k < i} \big[\big[f_k, [[f_j, f_i]]^1]]^1, f_{i+1}, \dots, f_{j-1}, \big[\big[f_k, [[f_j, f_i]]^1]]^2, f_{k+1}, \dots, [[f_j, f_i]]^2, \dots, f_{k-1}\\
    &\qquad + \sum_{i < l < j}\sum_{j < k < i} [[f_k, f_l]]^1, f_{l+1}, \dots, [[f_j, f_i]]^1, \dots, f_{l-1}, [[f_k, f_l]]^2, f_{k+1}, \dots, [[f_j, f_i]]^2, \dots, f_{k-1}\bigg)
  \end{align}
  }

  We show now that this sum is equal to zero: 
  \begin{itemize}
    \item Line (1) is zero by involutivity of $[-,-]$ and $\delta$. 
    \item Lines (2) and (4) are zero by involutivity of $\rho$ and $\sigma$. 
    \item Lines (3) and (5) cancel out with lines (6) and (10) by the bimodule compatibility of $[[-,-]]$. 
    \item Line (7) is zero by appropriately switching the roles of $i, j$, and $k$ and applying the double Jacobi of $[[-,-]]$.
    \item Line (8) is zero by appropriately switching the roles of $i, j$, and $l$ and applying the double Jacobi of $[[-,-]]$.
    \item Line (11) is zero by appropriately switching the roles of $i, j$, and $l$ and applying the double Jacobi of $[[-,-]]$.
    \item Line (12) is zero by appropriately switching the roles of $i, j$, and $k$ and applying the double Jacobi of $[[-,-]]$.
    \item Lines (9) cancels out with line (13) by appropriately switching the roles of $i, j, k$ and $l$. 
  \end{itemize}

  Thus, $(W, \{-,-\}, \Delta)$ is involutive.

\end{proof}

\begin{proof}[Proof of Theorem \ref{weavingbialgebra}]

  The maps $\{-,-\}_1$ and $\{-,-\}_2$ are Lie brackets by following the proof of Theorem \ref{liebracket}. We now show that $\{-,-\} = \{-,-\}_1 + \{-,-\}_2$ is a Lie bracket. Let $a = (\alpha_1; \beta_1; f_1^1, g_1^1, \dots, f_n^1, g_n^1)$, $b = (\alpha_2; \beta_2; f_1^2, g_1^2 \dots, f_m^2, g_m^2)$, and $c = (\alpha_3; \beta_3; f_1^3, g_1^3, \dots, f_p^3, g_p^3)$. 

  The bracket clearly satisfies skew symmetry, so we just need to prove Jacobi. First, observe that 
  \setcounter{equation}{0}
  \begin{align}
    \{a, \{b, c\}\} + \{b, \{c, a\}\} + \{c, \{a,b\}\} &= \{a, \{b, c\}_1\}_1 + \{b, \{c, a\}_1\}_1 + \{c, \{a, b\}_1\}_1\\
    &+ \{a, \{b, c\}_2\}_1 + \{b, \{c, a\}_2\}_1 + \{c, \{a, b\}_2\}_1\\
    &+ \{a, \{b, c\}_1\}_2 + \{b, \{c, a\}_1\}_2 + \{c, \{a, b\}_1\}_2\\
    &+ \{a, \{b, c\}_2\}_2 + \{b, \{c, a\}_2\}_2 + \{c, \{a, b\}_2\}_2
  \end{align}

  Lines (1) and (4) are zero by theorem \ref{liebracket}. There are two types of terms that appear in lines (2) and (3). Let's look at these two types in $\{a, \{b, c\}_2\}_1$

  The first takes the form 
  \[\sum_{k=1}^m \sum_{j=1}^m \sum_{l=1}^p \rho(\alpha_1, f_k^2), \dots, [[g_j^2, g_l^3]]_2, \dots, [[g_j^2, g_l^3]]_1, \dots, f_{k-1}^2.\]
  This term also appears in $\{c, \{a, b\}_1\}_2$ as 
  \[\sum_{l=1}^p \sum_{k=1}^m \sum_{j=1}^m \rho(\alpha_1, f_k^2), \dots, [[g_l^3, g_j^2]]_1, \dots, [[g_l^3, g_j^2]]_2, \dots, g_{k-1}^2\]
  By applying skew symmetry of $[[-,-]]$ in the second equation, we see that these two terms cancel out. We have a similar situation when the $f_k^2$ term is replaced by $f_k^3$, and the canceling term appears in $\{b, \{c, a\}_1\}_2$. By appropriately cycling $a, b$, and $c$, this proves that all terms of this type cancel out. 

  The second type of term takes the form 
  \[\sum_{i=1}^n \sum_{k=1}^m \sum_{j=1}^m \sum_{l=1}^p [[f_i^1, f_k^2]]_1, \dots, [[g_j^2, g_l^3]]_1, \dots, [[g_j^2, g_l^3]]_2, \dots, [[f_i^1, f_k^2]]_2, \dots, g_{i-1}\]
  However, this term also appears in $\{c, \{a, b\}_1\}_2$ as 
  \[\sum_{l=1}^p\sum_{j=1}^m\sum_{i=1}^n \sum_{k=1}^m [[g_l^3, g_j^2]]_1, \dots, [[f_i, f_j]]_2, \dots, [[f_i, f_j]]_1, \dots, [[g_l^3, g_j^2]]_1, \dots, f_l^3\]
  By once again applying skew symmetry of $[[-,-]]$, we see that these two terms cancel out. We again have a similar situation when the $f_k^2$ term is replaced by $f_k^3$, and the canceling term appears in $\{b, \{c, a\}_1\}_2$. By appropriately cycling $a, b,$ and $c$, this proves that all terms of this type also cancel out. 

  These cancellations occur because there is no intertwining of $\{-,-\}_1$ and $\{-,-\}_2$. For $i = 1, 2$, $\{-,-\}_i$ only interacts nontrivially with elements of $L_i$ and $M_i$. Similarly, $\Delta_i$ only interacts nontrivially with elements of $L_i$ and $M_i$, and so a similar proof to the one above along with a reference to the proof of theorem \ref{liecobracket} shows that $\Delta = \Delta_1 + \Delta_2$ is also a Lie cobracket. 

  We now show $(V, \{-,-\}, \Delta)$ is a Lie bialgebra. Let $a, b$ be as before. We will show that 
  \begin{equation}\label{Vliebialgebra}\Delta_j(\{a, b\}_i) = (\{a, -\}_i \otimes 1 + 1 \otimes \{a, -\}_i)\Delta_j(b) + (\{-, b\}_i \otimes 1 + 1 \otimes \{-, b\}_i) \Delta_j(b)\end{equation}
  for $i, j \in \{1, 2\}$, not necessarily distinct. 

  When $i = j$, this is already true by a reference to Theorem \ref{liebialgebra}. 

  Now assume without loss of generality that $i = 1$ and $j = 2$. The bracket $\{a, b\}_1$ interacts nontrivially only with the elements of $L_1$ and $M_1$, and then $\Delta_2$ interacts nontrivially only with the elements of $L_2$ and $M_2$. Thus, it doesn't matter which order we apply the bracket or the cobracket, and so both sides of \ref{Vliebialgebra} are the same. 
    
\end{proof}

\vspace{0.5cm}

\nocite{*}
\bibliographystyle{alpha}
\bibliography{bibliography}

\end{document}